\documentclass{article}
\usepackage{amsmath,amsthm,amsfonts,amssymb,mathrsfs}
\usepackage{authblk}
\usepackage{subfigure} 
\usepackage{graphicx}
\usepackage{color}
\usepackage{caption}
\usepackage[utf8]{inputenc}
\usepackage{epstopdf}
\usepackage[all]{nowidow}
\usepackage[intoc]{nomencl}

\DeclareMathOperator\supp{supp}
\newcommand\R{\mathbb{R}}
\newcommand\C{\mathbb{C}}
\newcommand\N{\mathbb{N}}

\newcommand{\bigchi}{\scalebox{1.2}{$\chi$}}
\newcommand{\eps}[1]{#1_{\varepsilon}}

\newcommand{\D}{\mathrm{d}}
\newcommand{\integral}{\int\limits_{\Omega}}
\newtheorem{theorem}{Theorem}

\newtheorem{assumption}{Assumption}

\newtheorem{lemma}[theorem]{Lemma}
\newtheorem{prop}[theorem]{Proposition}

\newtheorem{remark}[theorem]{Remark}

\theoremstyle{remark}

\usepackage{delarray}
\usepackage{enumerate}
\usepackage{setspace}

\makenomenclature
\makeindex
\begin{document}

\title{Local asymptotic stability of a system of integro-differential equations describing clonal evolution of a self-renewing cell population under mutation}


\author[*]{Jan-Erik Busse}
\author[**]{S\'{i}lvia Cuadrado}
\author[*]{Anna Marciniak-Czochra}

\affil[*]{Institute of Applied Mathematics
Interdisciplinary Center for Scientific Computing (IWR) and
BIOQUANT Center, Heidelberg}
\affil[**]{Departament de Matem\`atiques, Universitat Aut\`onoma de Barcelona}
\affil[ ]{\small \texttt{jan-erik.busse@bioquant.uni-heidelberg.de, silvia@mat.uab.cat, Anna.marciniak@iwr.uni-heidelberg.de}}
\date{}

\maketitle

\begin{abstract}\noindent
In this paper we consider a system of non-linear integro-differential equations (IDEs) describing evolution of a clonally heterogeneous  population of malignant white blood cells (leukemic cells) undergoing mutation and clonal selection. We prove existence and uniqueness of non-trivial steady states and study their asymptotic stability. The results are compared to those of the system without mutation.  Existence of equilibria is proved by formulating the steady state problem as an eigenvalue problem and applying a version of the Krein-Rutmann theorem for Banach lattices. The stability at equilibrium is analysed using linearisation and the Weinstein-Aronszajn determinant which allows to conclude local asymptotic stability.

\end{abstract}

\section{Introduction}\label{chap:mutdef}
This paper is devoted to the analysis of a system of integro-differential equations (IDEs) describing clonal evolution of a self-renewing cell population. The population is heterogeneous with respect to the self-renewal ability of dividing cells, a  property that may change due to cancerous mutations. The model is applied to describe dynamics of acute leukemias, an important type of malignant proliferative disorders of the blood forming system. \\

Acute leukemias show a considerable inter-individual genetic heterogeneity and a complex genetic relationship among different clones, i.e. subpopulations consisting of genetically identical cells  \cite{Ding,Lutz}. 
Similarly to the healthy hematopoietic system, the leukemic cell bulk is maintained by cells with stem-like properties that can divide and give rise to progeny cells which either adopt the same cell route as the parent cell (undergo self-renewal) or differentiate to a more specialised cell type \cite{Bonnet,Hope,Lutz}. There exists  theoretical~\cite{StiehlMMNP,StiehlBaranHoMarciniakCR} and experimental~\cite{Jung,Metzeler,Wang} evidence suggesting that the self-renewal ability of leukemic stem cells has a significant impact on disease dynamics and patient prognosis~\cite{SMC-Review2017,SMC-Review2019}. Increased self-renewal
confers a competitive advantage on cancer cell clones by leading to aggressive expansion of both stem and non-stem cancer cells and can be responsible for the clonal selection observed in experimental and clinical data \cite{Ding,VanDelft}. The latter has been investigated using mathematical models of evolution of an arbitrary number of leukaemic clones coupled to a healthy cell lineage \cite{StiehlBaranHoMarciniak,busse2016}.  A mathematical proof of  clonal selection has been shown in Ref. \cite{busse2016} exploiting the analytical tractability of the model with a continuum of heterogeneous clones differing with respect to the stem cell self-renewal ability.  A similar result has been  recently obtained in Ref. \cite{LMCS} for an extended model with two-parameter heterogeneity with respect to cancer stem cell self-renewal fraction and proliferation rate. It was shown that while increased proliferation rates may lead to a rapid growth of respective clones, the long-term selection process is governed by increased self-renewal of the most primitive subpopulation of leukemic cells  \cite{StiehlBaranHoMarciniak,LMCS}.  Mathematical analysis of the model provided an understanding of the link between the observed selection phenomenon and the nonlocal mode of growth control in the model, resulting from description of different plausible feedback mechanisms. Moreover, comparison of patient data and numerical simulations of the model allowing emergence of new clones suggested that self-renewal of leukemic clones increases with the emergence of clonal heterogeneity \cite{StiehlBD}. An open question is whether a mutation process with phenotypic heterogeneity in the course of disease may change the observed selection effect. To address this question, we propose an extension of the basic clonal selection model from Ref. \cite{busse2016} to account for the process of mutations. \\

The model from Ref. \cite{busse2016} takes the form
\begin{align}\label{eq:sysold}
\left\{\begin{array}{llll}
		\frac{\partial}{\partial t}u(t,x) &= \left(\frac{2a(x)}						{1+k\rho_2(t)} - 1\right)pu(t,x),\\
		\frac{\partial}{\partial t}v(t,x) &= 2\left( 1- \frac{a(x)}					{1+k\rho_2(t)}\right)pu(t,x) - dv(t,x),\\
		u(0,x) &= u_0(x),\\
		v(0,x) &= v_0(x), 
		\end{array}\right. 
\end{align}

where $x\in\Omega\subset\R,  a\in C(\overline{\Omega}), p,k,d\in\R_+$ and $\rho_2(t) = \int_{\Omega} v(t,x)\,\D x$.\\

The model describes evolution of one healthy cell lineage and an arbitrary number of leukemic clones, where the structural variable  $x\in \Omega$ represents a continuum of possible cell clones (e.g. characterised by different gene expression levels) differing with respect to the self-renewal ability of dividing cells. We follow the convention that  $x=x_0\in \Omega$ corresponds to healthy cells whereas different leukemic clones are characterised by different values of $x \in \Omega\setminus\{x_0\}$. The description of cell differentiation within each cell line is given by a two-compartment version of the multi-compartment system established in \cite{Marciniak2009}, mathematically studied in \cite{StiehlMCM,Nakata,StiehlMMNP,Knauer} and applied to patient data in \cite{StiehlBMT,StiehlBaranHoMarciniakCR}. The model focuses on self-renewal of primitive cells $u(t,x)$ and their differentiation to mature healthy cells $v(t,x_0)$ or leukemic blasts $v(t,x)$, $x \in \Omega\setminus\{x_0\}$, which do not divide. The two-compartment architecture is based on a simplified description of the multi-stages differentiation process.  Dividing cells $u$  give rise to two progeny cells.  A progeny cell is either more specialized than the mother cell, i.e., it is differentiated, or it is a copy of the mother cell (the case of self-renewal). The proliferation rate is denoted by the constant $p$.  The function $a(x)$ describes the fraction (probability) of self-renewal of cells of clone $x$, where dependence on $x$ reflects the clonal heterogeneity.  The feedback signal that promotes the self-renewal of dividing cells is modelled using a Hill function $ \frac{1}{1 + k \rho_2(t)}$, where the parameter $k>0$ is related to the degradation rate of the feedback signal~\cite{Marciniak2009,StiehlMCM,StiehlMMNP}. This formula has been derived from a simple model of cytokine dynamics using a quasi-stationary approximation~\cite{Getto,Mikelic}, motivated by biological findings presented in \cite{Kondo,Layton,Shinjo}. Implementing other plausible regulation mechanisms led to a similar model dynamics that can reproduce the clinical observation \cite{StiehlSciRep,Wang}. 
\\

In \cite{busse2016} it has been shown that the solution $(u,v)$ of system \eqref{eq:sysold} converges weakly$^*$ in the space of positive Radon measures $\mathcal{M}^+(\Omega)$  to a measure with support contained in the set of maximal values of the self-renewal fraction function $a$. In particular, if the set of maximal values of $a$ consists solely of discrete points, the solution $(u,v)$ converges weakly$^*$ to a sum of Dirac measures. In \cite{LMCS}, it has been shown that a similar result holds for a model with multiple compartments under an additional assumption preventing Hopf bifurcation that may occur in the model with at least three-stages maturation structure \cite{Knauer}.\\

The purpose of this paper is to extend the clonal selection model \eqref{eq:sysold} to include a mutation process causing phenotypic heterogeneity with respect to $x$. There are essentially two ways to model mutations in a continuous setting.
First, it can be done by adding a diffusive term,  which may be introduced using a Laplacian,  as in Ref. \cite{kimura, burger,lorz,mirrahimi,diplom,ChisholmLorenziClairambault,ChisholmLorenzietal,Almeida}. Such models are based on the assumption that mutations can occur at all times within a cell's life cycle and are not limited to division, as it is the case in epigenetic modifications. Consequently, such mutations do not change the overall number of individuals and only affect their distribution with respect to the structure variable.  From a mathematical point of view the diffusive ansatz provides good properties of the obtained solution and  allows using the  library of methods for semi-linear parabolic equations. Nevertheless, as the references above indicate, there exists a difficulty with characterisation of the long-term behaviour. 
Alternatively, mutations can be modelled with an integral operator, see for instance \cite{BB,calsina2004,calsina05,CCDR,LLR,GLGL,lorz2013}. This approach  includes mutations that occur during proliferation, which seems biologically realistic in case of genetic mutations \cite[Chapter 10]{graw2015}, \cite[Chapter 9]{alberts2013}. In contrast to the diffusive ansatz, the integral kernel allows  to characterise the mutation process and, for example, to model jumps in the traits. Mathematical advantage of the integral operators is related to their compactness, while  a disadvantage is the non-local structure which makes the analysis more complicated. \\

Selection processes under mutation have been studied using both classes of mutation models, however using different methods to show convergence of the solutions, in the limit of small or rare mutations, to weighted Dirac measures. In case of a reaction-diffusion equation (RDE), the most common ansatz is to transform the RDE into a Hamilton-Jacobi equation, for which the viscosity solution provides the desired convergence result \cite{mirrahimi, diplom}. For scalar equations,  this transformation can be performed using the WKB approximation method \cite{diekmann2005}. In case of a system of equations, the latter ansatz strongly depends on the structure of the model, since it is necessary to transform a system into a scalar Hamilton-Jacobi equation. In contrast, if mutation is modelled by an integral term, appropriate mathematical tools are given by theory of positive semigroups and the infinite dimensional version of the Perron-Frobenius theorem.\\ 

In this paper, we focus on the effect of rare mutations taking place during proliferation \cite{alberts2013,graw2015}, and propose a model based on integro-differential equations where we assume that, during proliferation, a mutation occurs with probability $\varepsilon$.  We prove, under suitable hypotheses, existence of locally asymptotically stable steady states of the model, which converge, for $\varepsilon \rightarrow 0$, to a weighted Dirac measure located at the point of maximum fitness of the corresponding pure selection model \eqref{eq:sysold}. The mathematical tools applied to analysis of the selection-mutation model are based on the ones used in \cite{silvia1,silvia2} and extended here to a system of two phenotype-structured populations in a non-compact domain. Similar results can be shown for small mutations that occur independently of the proliferation process \cite{calsina05}.\\

The paper is structured as follows. In Section $2$ we introduce the model with mutations and justify its well-posedness. Section $3$ is devoted to existence and uniqueness of stationary solutions. Convergence of the steady states for the zero limit of mutation rate is studied in Section $4$. In Section $5$, we show local asymptotic stability of the stationary solution of the model. The paper is completed with two appendices providing technical proofs needed for the results in Section $5$ and Section $3$ respectively.

\section{Selection-mutation model and its assumptions}
We extend system  \eqref{eq:sysold}  to account for mutations described by an integral kernel operator and consider the following system of integro-differential equations,
\begin{eqnarray}\label{eq:mut:sys}
\left\{\begin{array}{lll}
	\frac{\partial}{\partial t} v_{\varepsilon}(t,x) &=& 2\left( 				1-\frac{a(x)}{1+k\eps{\rho}(t)} \right)pu_{\varepsilon}(t,x) 	- dv_{\varepsilon}(t,x), \\
	\frac{\partial}{\partial t} u_{\varepsilon}(t,x) &=& \left( 				\frac{2a(x)}{1+k\eps{\rho}(t)} 												-(1+\varepsilon)\right)pu_{\varepsilon}			(t,x) + 		\varepsilon p\integral\kappa(x,y)u_{\varepsilon}(t,y)\,\D y, \\
	 v_{\varepsilon}(0,x) &=& v^0(x),\\
	 	u_{\varepsilon}(0,x) &=& u^0(x),
\end{array}\right.
\end{eqnarray}
where
\[\eps{\rho}(t) = \integral \eps{v}(t,x)\,\D x.\]
As previously,  $u_{\varepsilon}(t,x)$ denotes the density of dividing cells structured with respect to the trait $x$ that represents the expression level of genes influencing self-renewal ability of the cells, while $v_{\varepsilon}(t,x)$ denotes the resulting mature cells for $x=x_0\in \Omega$ or leukemic blasts of clone $x$ for $x\in \Omega\setminus\{x_0\}$.
The growth terms describing self-renewal and differentiation of dividing cells, regulated by a nonlocal nonlinear feedback from all non-dividing cells are taken from model \eqref{eq:sysold} in Ref. \cite{busse2016}. Additionally, the model accounts for mutations that take place during proliferation at a rate $\varepsilon \in (0,1]$. If a mutation occurs, then the probability density that an individual with trait $y$ mutates into one with trait $x$ is denoted by $\kappa(x,y)$. \\

In the remainder of this paper, we make the following assumptions:
\begin{assumption}\label{ass:mut:2}
	\quad
	\begin{enumerate}
		\item $\Omega\subset\R$ is open and bounded. 
		\item $a\in C^1(\overline{\Omega})$ with $0<a(x)<1$ for all 					  $x\in\Omega$ and there exists $x_{*}\in \Omega$ such that 					  $a(x_{*})>\frac12$. Moreover, there exists a single point $\bar{x}$ where the maximal value of the self renewal function is attained, i.e. $\bar{x} = 												  \mathrm{argmax}_{x\in\overline{\Omega}} a(x), \; \bar{a}= 				  a(\bar{x})$.
		\item $u^0, v^0 \in L^1(\Omega)$ with $u^0,v^0 > 					  0$.
		\item $\kappa \in C(\overline{\Omega}\times \overline{\Omega})$ is 				  strictly positive and such that $\int_{\Omega} \kappa(x,y)\mbox{d}x=1$ .
		 
	\end{enumerate}
\end{assumption}

The proof of existence and uniqueness of a classical solution of system \eqref{eq:mut:sys} follows directly from the Banach space-valued version of Picard-Lindel\"of theorem, see \cite{pazy,zeidler}, in combination with  boundedness of total mass, which can be obtained similarly as in Ref. \cite{busse2016}.\\

{\bf Numerical observation.} Numerical simulations of the model suggest a selection effect, similar to that in the pure selection model \eqref{eq:sysold}. The difference is that, depending on the size of mutation frequency $\varepsilon$, we observe a distribution of the different cell clones around the one with the highest self-renewal fraction, see Figure \ref{fig:eigen:differenteps}. Convergence of the system to a solution concentrated around the most aggressive phenotype is a fast process and does not depend on initial data. The remainder of this paper is devoted to a rigorous proof of this observation.
\begin{figure}
\centering
\includegraphics[scale=0.3]{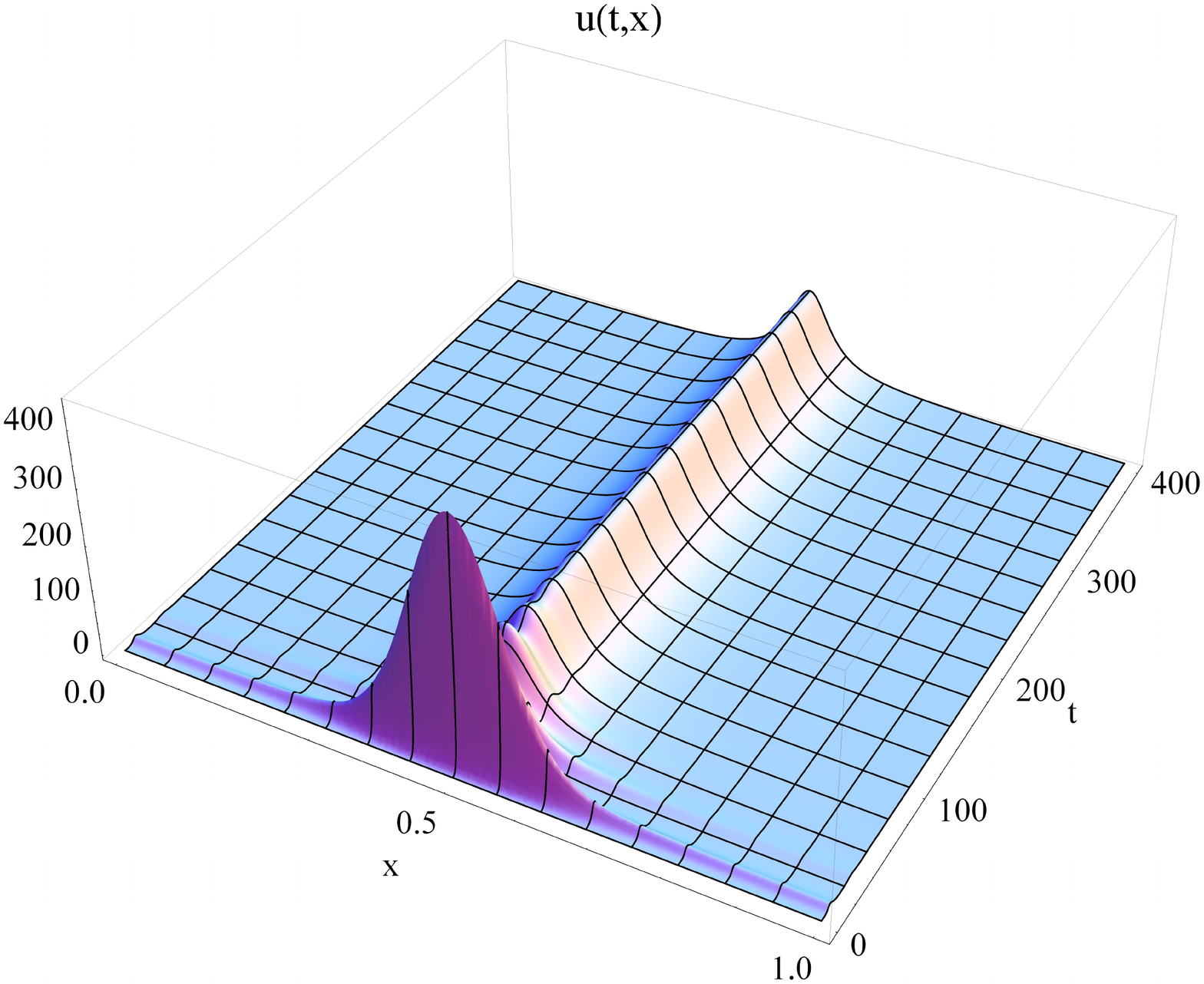}
\includegraphics[scale=0.3]{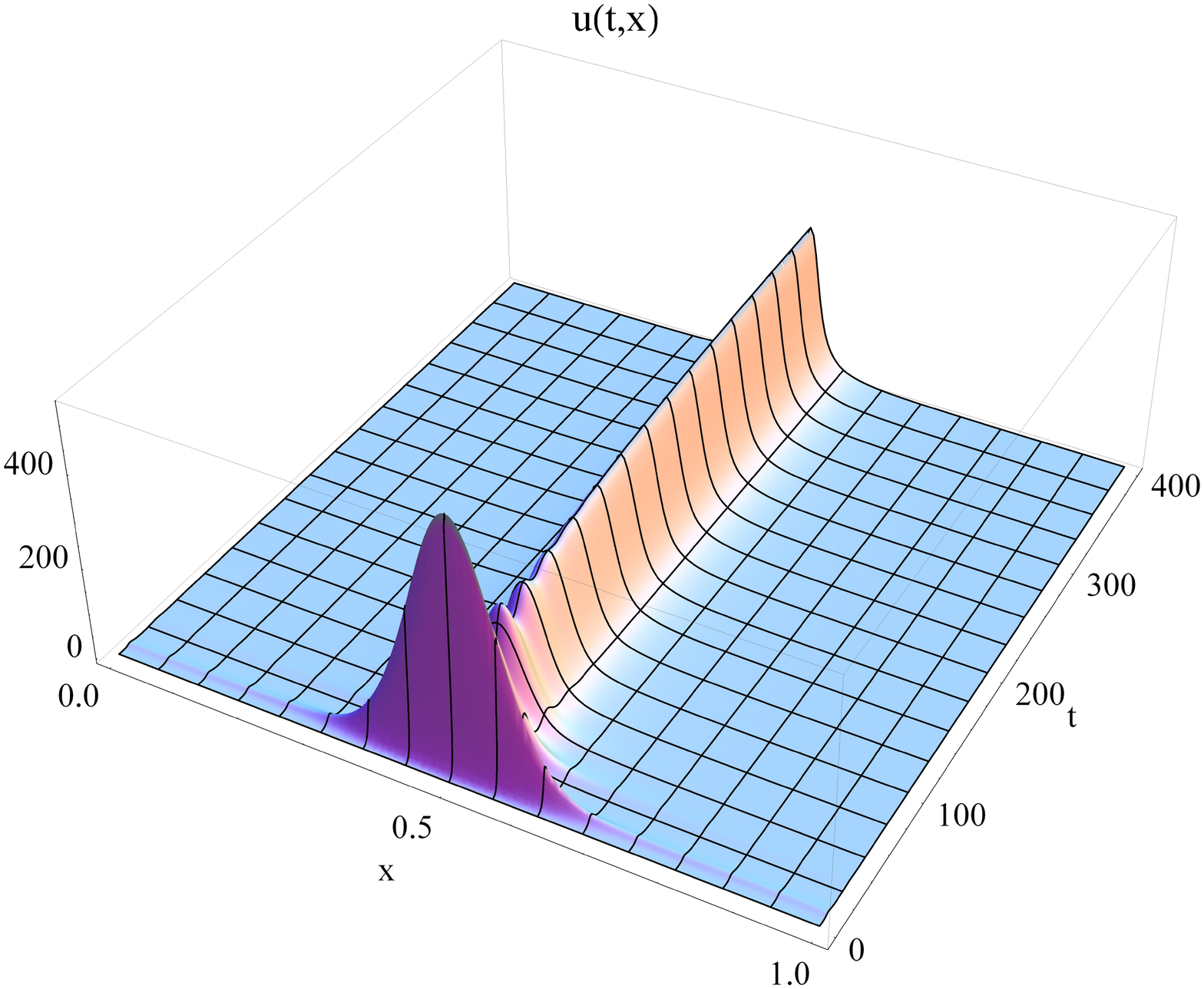}
\includegraphics[scale=0.3]{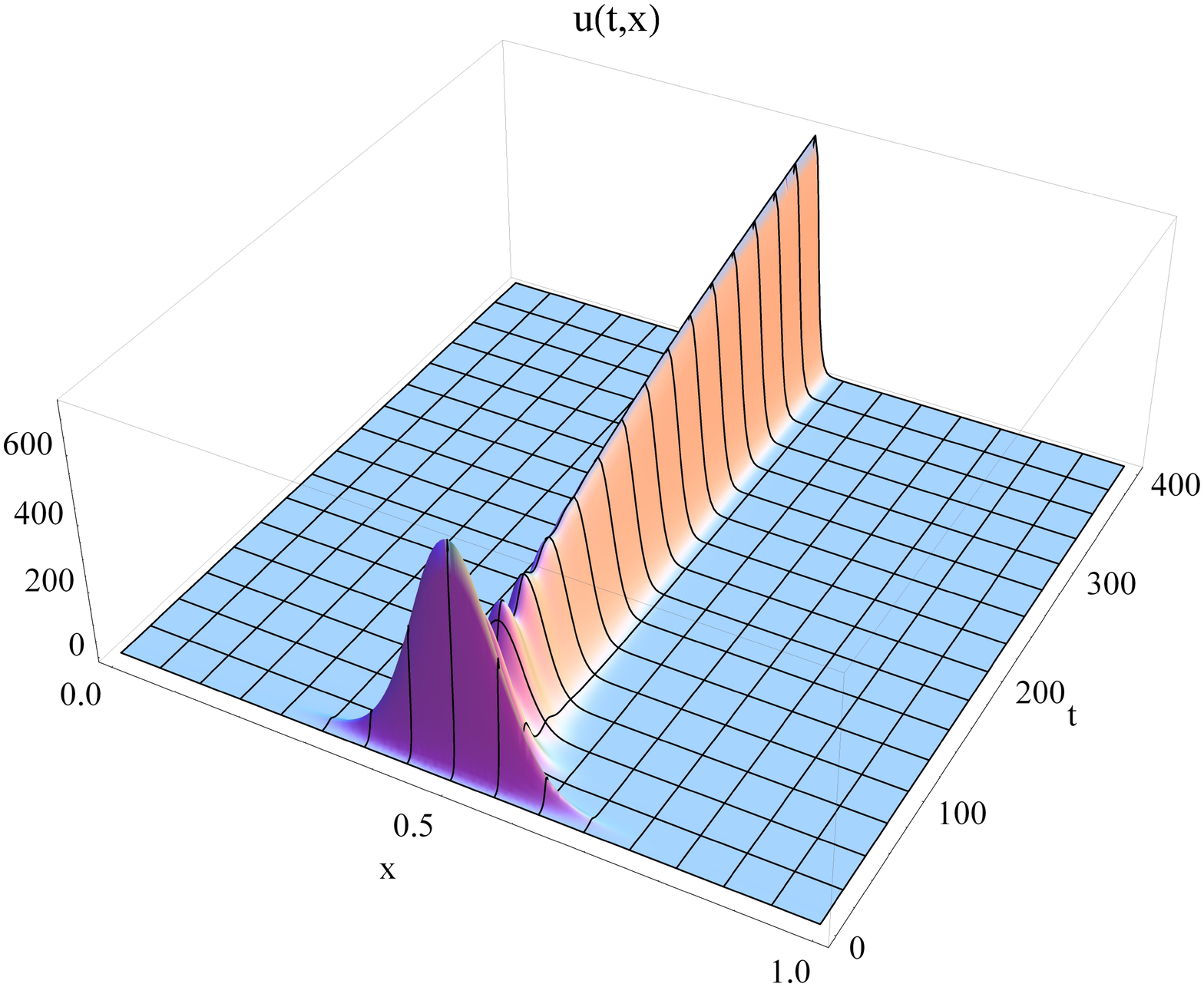}
\caption{The simulations depict the first solution $\eps{u}(t,x)$ for different values of $\varepsilon$. Going from top to bottom, the values of $\varepsilon$ are $\tfrac{3}{4}, \tfrac{1}{3}$ and $\tfrac{1}{100}$.  }
\label{fig:eigen:differenteps}
\end{figure}

\section{Existence and uniqueness of non-trivial steady states}\label{sec:mut:steady}

In general, the stationary problem for selection-mutation equations can be reduced (see \cite{calsina2004}) to a fixed point problem
for a real function whose definition depends on the existence and uniqueness of a dominant eigenvalue and a corresponding
positive eigenvector of a certain linear operator (obtained by fixing the nonlinearity in the model). To solve the problem for system \eqref{eq:mut:sys}, we follow an approach proposed in Ref. \cite{silvia1} for the stationary problem of a predator-prey model consisting of an IDE coupled with an ordinary differential equation (ODE). The structure of the ODE considered in \cite{silvia1} is a logistic type equation for which the steady state is given by a constant. Consequently, the steady state of the IDE depends on this constant that can be interpreted as a parameter. All together, the steady state problem can be reformulated as an eigenvalue problem, associated to the eigenvalue $0$, for which a positive eigenfunction is sought. The latter still depends on the parameter given by the steady state of the ODE.  To solve the coupled problem, it is necessary to choose the parameter in such a way that the eigenvalue problem and the steady state problem for the ODE are solved simultaneously. It results in solving a fixed point problem. In the remainder of this section, we adapt this approach to the cell population model \eqref{eq:mut:sys} which consists of a system of two IDE's.

\subsection{Eigenvalue problem}\label{subsec:mut:steady}

To find steady states of model \eqref{eq:mut:sys}, we consider the model obtained by integrating the first equation in \eqref{eq:mut:sys}. %
\begin{eqnarray}\label{eq:mut:sys2}
\left\{\begin{array}{lll}
	\frac{\D}{\D t} \rho_{\varepsilon}(t) &=& \integral 2\left( 				1-\frac{a(x)}{1+k\rho_{\varepsilon}(t)} \right)pu_{\varepsilon}				(t,x)\,\D x - d\rho_{\varepsilon}(t), \\
	\frac{\partial}{\partial t} u_{\varepsilon}(t,x) &=& \left( 				\frac{2a(x)}{1+k\rho_{\varepsilon}(t)} 										-(1+\varepsilon)\right)pu_{\varepsilon}	(t,x) + 				\varepsilon p\integral\kappa(x,y)u_{\varepsilon}(t,y)\,\D y, \\
	u_{\varepsilon}(0,x) &=& u^0(x),\\
	 \rho_{\varepsilon}(0) &=& \integral v^0(x)\,\D x.
\end{array}\right.
\end{eqnarray}
Since the nonlinearity depends only on the total population of mature cells (the integral of the second variable), the integrated equation becomes an ordinary differential equation for $\rho_{\varepsilon}(t)$. Consequently, the first component of a steady state $(\rho_{\varepsilon}, u_{\varepsilon}(x))$ of system \eqref{eq:mut:sys2} is a constant. Furthermore, the first component of the (corresponding) steady state of system \eqref{eq:mut:sys} can be computed  by inserting the steady state of system \eqref{eq:mut:sys2} $(\rho_{\varepsilon}, u_{\varepsilon}(x))$ into the first equilibrium equation of \eqref{eq:mut:sys} and solving it for $v_{\varepsilon}$.\\

The (nontrivial) equilibria of system \eqref{eq:mut:sys2} are given by the solutions of 
\begin{eqnarray}\label{eq:mut:steady}
\left\{\begin{array}{ll}
	0 =& \integral 2\left( 1-\frac{a(x)}{1+k\rho_{\varepsilon}} 				\right)pu_{\varepsilon}(x)\,\D x - d\rho_{\varepsilon}, \\
	0 =& \left(\frac{2a(x)}{1+k\rho_{\varepsilon}} 								-(1+ \varepsilon)\right)pu_{\varepsilon}(x) + 					\varepsilon p \integral\kappa(x,y)u_{\varepsilon}(y)\,\D y.
\end{array}\right.
\end{eqnarray}
Let us define for $\eps{\rho}>0$,

\begin{align}
B_{\varepsilon,\eps{\rho}}&: L^1(\Omega) \rightarrow L^1(\Omega), &
B_{\varepsilon,\eps{\rho}}\eps{u}(x) &:=  \left(\frac{2a(x)}{1+k\eps{\rho}} -(1+ \varepsilon)\right)p u_{\varepsilon}(x),\nonumber\\
\eps{K}&: L^1(\Omega) \rightarrow L^1(\Omega), &
\eps{K}\eps{u}(x) &:= \varepsilon p\integral \kappa(x,y)u_{\varepsilon}(y)\,\D y,\nonumber\\
C_{\varepsilon,\eps{\rho}}&: L^1(\Omega) \rightarrow L^1(\Omega), &
C_{\varepsilon,\eps{\rho}}u_{\varepsilon} &:= B_{\varepsilon,\eps{\rho}} u_{\varepsilon} + K_{\varepsilon} u_{\varepsilon}\label{operators}. 
\end{align}
If a non-trivial steady state of system \eqref{eq:mut:sys2} exists, the first equation of system \eqref{eq:mut:steady} provides a constant solution $\rho \in (0,\infty)$. The second equation of system \eqref{eq:mut:steady} can be then interpreted as an eigenvalue problem for $C_{\varepsilon,\rho}$, which depends on the parameter $\rho$. Thus, we are looking for a function $\varphi_{\varepsilon,\rho}$ and a constant $\eps{\lambda}(\rho)$ such that
\begin{eqnarray}\label{eq:mut:eigen}
C_{\varepsilon,\rho} \varphi_{\varepsilon,\rho} = \lambda_{\varepsilon}(\rho) \varphi_{\varepsilon,\rho}.
\end{eqnarray}
The first component of the steady state of system \eqref{eq:mut:sys2} is then given by the solution $\rho_{\varepsilon}$ of the equation $\lambda_{\varepsilon}(\rho)=0$. Denoting by $\varphi_{\varepsilon,\eps{\rho}}$ the corresponding normalized eigenfunction, the second component of the steady state of \eqref{eq:mut:sys2} has the form
$\eps{u}= c_{\varepsilon}\varphi_{\varepsilon,\rho}$ where $c_{\varepsilon}\in (0,\infty)$ satisfies
\[c_{\varepsilon} = \frac{d\rho_{\varepsilon}}{\integral 2\left( 1 - \frac{a(x)}{1+k\rho_{\varepsilon}}\right)p \varphi_{\varepsilon,\rho_{\varepsilon}}(x)\,\D x}.\]
Let us observe that $0$ is an eigenvalue with corresponding eigenfunction $\varphi_{\varepsilon,\rho}$ of $C_{\varepsilon,\rho}$ if and only if
\begin{alignat*}{2}
	& &  B_{\varepsilon,\rho}\varphi_{\varepsilon,\rho} + K_{\varepsilon} 					\varphi_{\varepsilon,\rho} &= 0, \\
	&\Leftrightarrow &  K_{\varepsilon} \varphi_{\varepsilon,\rho} &= 			-B_{\varepsilon,\rho}\varphi_{\varepsilon,\rho}, \\
	&\Leftrightarrow &   K_{\varepsilon}\left(- 								B_{\varepsilon,\rho}^{-1}\psi_{\varepsilon,\rho}\right) &= \psi_{\varepsilon,			\rho}
\end{alignat*}
with $\psi_{\varepsilon,\rho} := -B_{\varepsilon,\rho}\varphi_{\varepsilon,\rho}$. This means that $\psi_{\varepsilon,\rho}$ is an eigenfunction corresponding to eigenvalue $1$ of the operator $T_{\varepsilon,\rho} : L^1(\Omega) \rightarrow L^1(\Omega)$ given by
\begin{align}\label{def:mut:t}
\begin{aligned}
T_{\varepsilon,\rho}u &:=K_{\varepsilon}\circ (-B_{\varepsilon,\rho}^{-1}u)\\ 
& =\varepsilon\integral \kappa(x,y) \frac{1+k\rho}{(1+k\rho)(1+ \varepsilon)-2a(y)} u(y)\,\D y.
\end{aligned}
\end{align}

\begin{remark}
The equivalence of the eigenvalues $1$ of the operator $T_{\varepsilon,\rho} (=K_{\varepsilon}\circ (-B_{\varepsilon,\rho}^{-1}))$  and $\eps{\lambda}(\rho)$ of the operator $C_{\varepsilon,\rho}(=B_{\varepsilon,\rho}+K_{\varepsilon})$ is exploited in \cite{silvia1}, but the idea goes back to \cite[Proposition $2.1$]{burgerpert}. 
We choose to study the eigenvalue problem for the operator $T_{\varepsilon,\rho}$, because it allows a direct application of the version for Banach lattices of the well-know Krein-Rutmann theorem (\cite{daners92}).\\
Alternatively, one can directly study the eigenvalue problem for the operator $C_{\varepsilon,\rho}$ that would require more work. Existence of a strictly dominant eigenvalue of $C_{\varepsilon,\rho}$ can be obtained from an application of a result of Greiner (Corollary 1.8 in Ref. \cite{greiner}) that provides existence of an algebraically simple, strictly dominant eigenvalue
of a perturbation $B + K$ of the generator of a positive semigroup by a positive bounded irreducible operator $K$ satisfying that there exists and integer
$n$ such that $(KR(\lambda, B))^{n}$ is compact for all $\lambda$ with $\text{Re}\lambda > s(B)$ and that
$s(B + K) > s(B)$.
\end{remark}

We proceed then to investigate the eigenvalue problem for the operator $T_{\varepsilon,\rho}$.
We define
\begin{equation}
\label{def:ro}
\bar{\rho}_{\varepsilon}:=\frac{1}{k}\left(\frac{2\bar{a}}{1+\varepsilon}-1\right)
\end{equation}
where $\bar{a}=\max_{x} a(x)=a(\bar{x})$.
Notice that $T_{\varepsilon,\rho}$ is a positive operator for $\rho \in (\bar{\rho}_{\varepsilon},+\infty)$.
\begin{prop}\label{prop:mut:eigenvalue}
Let $T_{\varepsilon,\rho}$ be the linear operator defined in \eqref{def:mut:t} for $\rho\in (\bar{\rho}_{\varepsilon},+\infty)$. Its spectral radius $r(T_{\varepsilon,\rho})$ is an algebraically simple eigenvalue of $T_{\varepsilon,\rho}$ with a corresponding strictly positive eigenfunction. Moreover, $r(T_{\varepsilon,\rho})$ is the only eigenvalue of $T_{\varepsilon,\rho}$ having a positive eigenfunction. 
\end{prop}
\begin{proof}
By the Krein-Rutman theorem for Banach lattices, see \cite[Theorem 			$12.3$]{daners92}, the problem reduces to proving that $T_{\varepsilon,\rho}$ is a compact positive irreducible operator.\\

	$T_{\varepsilon,\rho}$ is a positive operator by definition and the choice of $\rho$. It is also evident that $T_{\varepsilon,\rho}$ is a bounded operator. Defining
	\[\tilde{\kappa}(x,y):= \kappa(x,y) \frac{1+k\rho}							{(1+k\rho)(1+\varepsilon)-2a(y)},\]
	we obtain that $\tilde{\kappa}>0$  for $\rho \in (\bar{\rho}_{\varepsilon},+\infty)$ and $\tilde{\kappa}$ is continuous on 				$\overline{\Omega}^2$. Hence, 	$T_{\varepsilon,\rho}: L^1(\Omega) 			\rightarrow L^1(\Omega)$ is irreducible due to \cite[Chapter $V$, \S 	$6$, Example $4$]{schaefer}. 
	
	Additionally, according to 					\cite[Corollary $5.1$]{eveson}, for
	\[\bar{\kappa}(x,y) := \left\{\begin{array}{ll}
									\tilde{\kappa}(x,y), & y\in\Omega\\
									0, & y\notin \Omega
									\end{array}\right. ,\]
	 $T_{\varepsilon,\rho}: L^1(\Omega) 		\rightarrow L^1(\Omega)$ is compact if and only if for all $\iota>0$, 	there exist 	$\delta>0, R>0$ such that for almost all 					$x\in\Omega$ and 			for every 		$h\in\R$ with $\left| 		h\right| < \delta$ 
	\begin{align*}
    	\int\limits_{\R\setminus B_R(0)} 					\left| 					\bar{\kappa}(x,y)\right|\,\D y < \iota, \quad 								\int\limits_{\R} \left| \bar{\kappa}(x,y+h) - \bar{\kappa}				(x,y)\right| \,\D y < \iota.
	\end{align*}
	Let $\iota>0$ be arbitrary but fixed.

	As $\Omega$ is bounded, let us choose $R>0$ such that
	\begin{equation*}
	\left| 						\Omega\setminus B_R(0)\right|< 				\frac{\iota}					{\max\limits_{(x,y)\in\overline{\Omega}^2}\tilde{\kappa}(x,y)}.
	\end{equation*}
	 Then,
	\[\int\limits_{\R\setminus B_R(0)} \left| \bar{\kappa}(x,y)\right|		\,\D y = \int\limits_{\Omega\setminus B_R(0)} \left| \tilde{\kappa}			(x,y)\right|\,\D y \leq \max\limits_{(x,y)\in\overline{\Omega}^2} 			\tilde{\kappa}(x,y)\left| \Omega\setminus B_R(0) \right| <	\iota.\]
	Due to the dominated convergence theorem and the continuity of 				$\tilde{\kappa}$, it holds
	\[\int\limits_{\R} \left| \bar{\kappa}(x,y+h) - \bar{\kappa}				(x,y)\right|\,\D y = \integral \left| \tilde{\kappa}(x,y+h)- 				\tilde{\kappa}(x,y)\right|\,\D y < \iota,\]
	for $\left| h\right|$ small enough, which completes the proof. \\
\end{proof}
From the previous proposition we have that the operator $T_{\varepsilon,\rho}$ admits a strictly positive eigenfunction corresponding to the eigenvalue $r(T_{\varepsilon,\rho})$. What is left to show, in order to obtain equilibria, is that it is possible to choose $\rho$ such that $r(T_{\varepsilon,\rho}) = 1$ and that this choice of $\rho$ is unique (remember that showing $r(T_{\varepsilon,\rho})=1$ is equivalent to showing $\eps{\lambda}(\rho) = 0$). \\
The idea is to prove that $r(T_{\varepsilon,\rho})$ is continuous with respect to $\rho$ and strictly monotone. This ansatz goes back to \cite{burgerpert}.
\begin{lemma}\label{lem:mut:specmon}
Let $T_{\varepsilon,\rho}$ be the linear operator defined in \eqref{def:mut:t} for $\rho\in (\bar{\rho}_{\varepsilon},+\infty)$. Its spectral radius, $r(T_{\varepsilon,\rho})$	is a continuous function of $\rho$. Moreover, there exists $\varepsilon_0$ such that for $\varepsilon<\varepsilon_0$,  $r(T_{\varepsilon,\rho})$ is strictly decreasing.
\end{lemma}
%
%\begin{remark}\label{rem:mut:conti}
%In general the spectrum of an operator does not need to be 					continuously dependent on the operator itself. It is possible to show that for a linear operator $T$ the spectrum $\sigma(T)$ depends upper semi-continuously on $T$, compare \cite[Theorem $3.1$, p. 208]{kato}. The upper semi-continuity prevents the spectrum to expand arbitrarily fast, but it cannot prevent it from decaying, see for instance \cite[Example $3.8$]{kato}. If it is sufficient to have a look at a finite system of isolated eigenvalues, continuity can be achieved if the investigated operator $T$ is closed, see \cite[Chapter $IV$, \S $3.5$, p. 213]{kato}.
%\end{remark}
%
\begin{proof}[Proof of Lemma \ref{lem:mut:specmon}]
Continuity of $r(T_{\varepsilon,\rho})$ with respect to $\rho$ follows from the continuity of a finite system of eigenvalues of a closed operator (\cite[Chapter IV, \S $3.5$]{kato}).\\
	For the monotonicity, we use Gelfand's formula for the spectral radius of a bounded linear operator $A$ on a Banach space
	\[r(A) = \lim\limits_{n\rightarrow\infty} \left\| A^n\right\|				_{\infty}^{\frac1n},\]
	where $\left\|.\right\|_{\infty}$ is the operator norm. We have to 			show that
	\[\left\| T_{\varepsilon,\rho_1}^n \right\|_{\infty}^{\frac1n} > 			\left\| T_{\varepsilon,\rho_2}^n\right\|_{\infty}^{\frac1n}\;				\mbox{for}\; \rho_1<\rho_2.\]
	A straightforward proof by induction provides the formula 
	\[ \begin{array}{rcl}T^n_{\varepsilon,\rho} u(x)& =& \int\limits_{\Omega^n} 						\kappa(x,y_n)\prod\limits_{i=1}^{n-1} 										\kappa(y_i,y_{i+1})\\ & & \frac{(1+k\rho)^n}{\prod\limits_{i=1}^n 					((1+k\rho)(1+ \varepsilon)-2a(y_i))}u(y_n)\,\D y_1 			\cdots \D y_n.\end{array}\]
	In order to obtain monotonicity, we compute the derivative 					of 	$T^n_{\varepsilon,\rho}$ with respect to $\rho$. Another 				straightforward proof by induction (see Appendix B) shows that $\tfrac{\D}					{\D\rho}T^n_{\varepsilon,\rho}<0$ for $\varepsilon$ small enough. Thus $T_{\varepsilon,\rho_1}^n u > T_{\varepsilon,			\rho_2}^nu$ for all $u\in L^1(\Omega), u\geq 0$ and $\rho_1< 				\rho_2$. Then, taking the operator norm on both 	sides and using 		that the function $x\mapsto x^{\frac1n}$ is strictly monotone, we 			obtain
	\[\left\| T_{\varepsilon,\rho_1}^n \right\|_{\infty}^{\frac{1}{n}} \geq \left\| T_{\varepsilon,				\rho_2}^n\right\|_{\infty}^{\frac1n} \;\mbox{for all}\; n\in\N \Rightarrow r(T_{\varepsilon,\rho_1}) \geq r(T_{\varepsilon,			\rho_2})\quad\mbox{for}\; \rho_1<\rho_2.\]
 	The argument for strict monotonicity is the same 		as in 			Ref.	\cite{silvia1}. 
\end{proof}
Up to this point we have showed that the spectral radius $r(T_{\varepsilon,\rho})$ is a continuous and strictly decreasing function of $\rho$. Hence it is necessary to prove that there exists some $\rho \in (\bar{\rho}_{\varepsilon},\infty)$ such that $r(T_{\varepsilon,\rho})=1$. This is provided by 
\begin{lemma}\label{lem:mut:unique}
	For all $\varepsilon<\varepsilon_0$, there exists a unique 					$\rho\in(\bar{\rho}_{\varepsilon},\infty)$ such that $$r(T_{\varepsilon,\rho})=1.$$ 
\end{lemma}
\begin{proof}
	 We 		observe that
	\[ \lim\limits_{\rho\rightarrow \infty} T_{\varepsilon,\rho} = \frac{\varepsilon}											{1+\varepsilon}\integral 										\kappa(x,y)\eps{u}(y)\,\D y := 				T_{\varepsilon}.\]
	$T_{\varepsilon}$ is a bounded operator, hence the spectrum is bounded.  Since $r(T_{\varepsilon})\leq \left\| 						T_{\varepsilon}\right\|_{\infty}$ we obtain $r(T_{\varepsilon}) 			<1$ . As $r(T_{\varepsilon,\rho})$ 			is	continuous with respect to $\rho$, we can find $\rho_1\in 				(\bar{\rho}_{\varepsilon},\infty)$ such that $r(T_{\varepsilon,\rho_1})<1$. 
	On the other hand, since \begin{eqnarray*}
		\left\| T_{\varepsilon,\rho}\right\|_{\infty} &=& 							\sup\limits_{\substack{u\in L^1(\Omega)\\ \left\| u\right\|					_{L^1(\Omega)} =1}} \left| 					\varepsilon 					\integral \kappa(x,y) \frac{1+k\rho}{((1+k\rho)(1+ 							\varepsilon)-2a(y))} u(y)\,\D y	\, \D x\right| 
	\end{eqnarray*}
	and $r(T_{\varepsilon,\rho}) = \lim\limits_{n\rightarrow\infty} \left\|T_{\varepsilon,\rho}^{n}\right\|^{\frac{1}{n}}$ we have that $\lim_{\rho \to \bar{\rho}_{\varepsilon}}r(T_{\varepsilon,\rho})= +\infty$.  Lemma \ref{lem:mut:specmon} and the intermediate value theorem imply the statement.
\end{proof}
We can now formulate the theorem giving existence and uniqueness of steady states of system \eqref{eq:mut:sys}.
\begin{theorem}\label{thm:mut:existe}
	There exists some $\varepsilon_0>0$ such that for all $\varepsilon<\varepsilon_0$ there exists a 	unique, non-trivial  steady state $(v_{\varepsilon},u_{\varepsilon})$ of system \eqref{eq:mut:sys}.
\end{theorem}
\begin{proof}
	Combination of Proposition \ref{prop:mut:eigenvalue} and Lemma \ref{lem:mut:unique} yields the existence of a unique, non-trivial steady state $(\rho_{\varepsilon},u_{\varepsilon})$ of system \eqref{eq:mut:sys2}. Substituting it in the first equilibrium equation in system \eqref{eq:mut:sys}, we obtain the first component of its unique, non-trivial steady state $(v_{\varepsilon},u_{\varepsilon})$.
\end{proof}
\section{Convergence of the steady states}

Once we have proved, for $\varepsilon$ small enough, existence of a stationary solution $(v_{\varepsilon},u_{\varepsilon})$ of system \eqref{eq:mut:sys}, we are interested in the behavior of this steady state when the mutation rate goes to zero.\\

 In order to study it, we recall the equivalence between the eigenvalue problem for the operators $C_{\varepsilon,\rho}=B_{\varepsilon,\rho}  + K_{\varepsilon}$ and $T_{\varepsilon,\rho}=K_{\varepsilon} (B_{\varepsilon,\rho})^{-1}$ which is 
\begin{eqnarray*}
	B_{\varepsilon,\rho} \varphi_{\varepsilon,\rho} + K_{\varepsilon} 			\varphi_{\varepsilon,\rho} &=& \lambda_{\varepsilon}(\rho) 					\varphi_{\varepsilon,\rho} \\
	\Leftrightarrow K_{\varepsilon} (\lambda_{\varepsilon}(\rho) - 			B_{\varepsilon,\rho})^{-1}\psi_{\varepsilon,\rho} &=& 						\psi_{\varepsilon,\rho},
\end{eqnarray*}
for $\psi_{\varepsilon,\rho} = (\lambda_{\varepsilon}(\rho) - B_{\varepsilon,\rho})\varphi_{\varepsilon,\rho}$. That is, $\varphi_{\varepsilon,\rho}$ is an eigenfunction of 
 $C_{\varepsilon,\rho}$ corresponding to eigenvalue $\eps{\lambda}(\rho)$ if and only if $1$ is an eigenvalue of $K_{\varepsilon} (\lambda_{\varepsilon}(\rho) - 			B_{\varepsilon,\rho})^{-1}$ with eigenfunction $
\psi_{\varepsilon,\rho}$.\\

The proof of convergence of the steady state consists of the following steps: We begin by showing that $\lambda_{\varepsilon}(\rho)$ is a strictly dominant eigenvalue of $C_{\varepsilon,\rho}$ with a corresponding strictly positive eigenfunction. Then we show that for $\varepsilon\rightarrow 0$,  $\lambda_{\varepsilon}(\rho) \rightarrow \max_{x\in \overline{\Omega}}\left(\frac{2a(x)}{1+k\rho}-1\right)p$ pointwise, from which we conclude convergence of the corresponding eigenfunctions $\varphi_{\varepsilon,\rho}$. Additionally, convergence of the eigenvalues $\eps{\lambda}(\rho)$ allows deducing convergence of zeros of the eigenvalues, $\eps{\rho}$, and ultimately convergence of the steady state.

\subsection{Existence of eigenvalues}
Following \cite[Theorem $2.2$]{burgerpert}, in order to show that $\lambda_{\varepsilon}(\rho)$ is a strictly dominant eigenvalue of $C_{\varepsilon,\rho}$ with a corresponding strictly positive eigenfunction,
 it is sufficient to find some $$\lambda_1 > s(B_{\varepsilon,\rho})=\max\limits_{x\in\overline{\Omega}} \left( 					\frac{2a(x)}{1+k\rho} - (1+ \varepsilon)\right)p$$ such that $r(K_{\varepsilon} (\lambda_1Id - B_{\varepsilon,\rho})^{-1})>1$, (where $s(A)$ and $r(A)$ denote respectively the spectral bound and the spectral radius of a linear operator $A$).

We use the following characterization of the spectral bound of the generator $A$ of a strongly continuous positive semigroup
\begin{equation}
\label{car}
s(A) \geq \sup\{ \mu \in \mathbb{R} \quad : \quad Af \geq \mu f \quad \text{for some} \quad 0 < f \in D(A)\}.
\end{equation}
This property of the spectral bound is stated in Ref. \cite{nagel} in $C(K)$, the space of all
real-valued continuous functions on a compact space $K$, but the proof also holds for any generator of a positive semigroup in a Banach lattice such that the spectral bound and the growth bound coincide, which is the case in  $L^{p}$-space, $1 \leq p < \infty$, \cite{weis}.

\begin{prop}\label{prop:mut:existe2}
	There exists $\lambda_1 > s(B_{\varepsilon,\rho})$ such that 				$r(K_{\varepsilon}(\lambda_1Id - B_{\varepsilon,\rho})^{-1})>1$. 
\end{prop}
\begin{proof}
	Let us define function $q: \R\times \Omega \rightarrow \R$ 		by 
	\[q(\lambda,y) := \lambda - \left( \frac{2a(y)}{1+k\rho} -(1+ 				\varepsilon)\right)p.		\]
	Then, we need to prove that
	\[\exists \lambda_1> s(B_{\varepsilon,\rho}) \;\exists g\in L^1(\Omega), g> 0\;			\forall 			x\in\Omega: \quad \varepsilon \integral 				\kappa(x,y)q(\lambda_1,y)^{-1} 	g(y)\,\D y > g(x).\]
	Observe that $\mathrm{argmin}_{x\in\Omega} q(s(B_{\varepsilon,\rho					}),x) = \mathrm{argmax}_{x\in\Omega} \left( \frac{2a(x)}			{1+k\rho}-(1+ \varepsilon)\right)p = \bar{x}$. \\ Using the definition of $q$, we 			obtain
	\begin{eqnarray}\label{eq:mut:pprop}
		\left\{\begin{array}{ll}
				q(s(B_{\varepsilon,\rho}),\bar{x}) &= 0, \\
				\frac{\partial}{\partial y} q(s(B_{\varepsilon,\rho}),y)_{|								y=\bar{x}} &= 0.
				\end{array}\right. 
	\end{eqnarray}	
	Choosing  $g = \bigchi_{B_{\delta}(\bar{x})}$, for a small $\delta>0$, we estimate
	\begin{eqnarray}\label{ineq:mut:p}
		\varepsilon \integral \kappa(x,y) q(\lambda_1,y)^{-1}g(y)\,\D y 			&=& \varepsilon \int\limits_{B_{\delta}(\bar{x})} \kappa(x,y) 				q(\lambda_1,y)^{-1}\,\D y \nonumber\\
		&\geq& \varepsilon\min\limits_{x,y\in\overline{\Omega}} 					\kappa(x,y)\int\limits_{B_{\delta}											(\bar{x})}q(\lambda_1,y)^{-1}\,\D y.
	\end{eqnarray}	
	Expanding the function $q$ using the Taylor formula up to the $0$th 	order around $\bar{x}$ yields
	\begin{eqnarray}\label{eq:mut:taylor}
		q(\lambda_1,y) = q(\lambda_1,\bar{x}) + 									o\left( \left\| y-\bar{x}\right\| \right).
	\end{eqnarray}
	Inserting equation 			\eqref{eq:mut:taylor} into inequality 			\eqref{ineq:mut:p} leads to
	\begin{eqnarray*}
		\varepsilon\integral \kappa(x,y) q(\lambda_1,y)^{-1} g(y)\,\D y 			&\geq& \varepsilon\min\limits_{x,y\in\overline{\Omega}} 					\kappa(x,y) \int\limits_{B_{\delta}(\bar{x})} \frac{1}						{q(\lambda_1,\bar{x}) + o\left( \left\| y-\bar{x}\right\|					\right)} \,\D y  >1,
	\end{eqnarray*}
	by using the first equation of \eqref{eq:mut:pprop}, choosing 				$\delta$ small enough and $\lambda_1$ close enough to 						$s(B_{\varepsilon,\rho})$. 
\end{proof}
\subsection{Convergence of the eigenvalues and the eigenfunctions}
Now that we have proved existence of a strictly dominant eigenvalue $\lambda_{\varepsilon}(\rho)$ of the operator $C_{\varepsilon,\rho}$
we can formulate a result about its limiting behavior.
\begin{lemma}\label{lem:mut:eigenconv}
	Let $\lambda_{\varepsilon}(\rho)$ be the strictly dominant 					eigenvalue of the operator $C_{\varepsilon,\rho}$ defined in \eqref{operators}, then 
	\[\lambda_{\varepsilon}(\rho) \stackrel{\varepsilon\rightarrow 0}			{\longrightarrow} \max\limits_{x\in\overline{\Omega}} \left( 				\frac{2a(x)}{1+k\rho}-1\right)p.\]
\end{lemma}	
\begin{proof}
	For notational simplicity, we denote $\mu(x) := 							\left( \frac{2a(x)}{1+k\rho} -1\right)p,\\
	 \eps{\mu}(x) := \left( 			\frac{2a(x)}{1+k\rho} -(1+\varepsilon)\right)p$. We want to show that for 			all $\delta>0$ there exists 					$\varepsilon_0$ such that for all 											$\varepsilon<\varepsilon_0$ it holds 
	\[\eps{\lambda}(\rho) \in B_{\delta}(\mu(\bar{x})),\]
	where recall that $\bar{x}$ denotes the point where the maximal value of the self-renewal function is attained.\\
	We begin by proving that $\lambda_{\varepsilon}(\rho)\leq \mu(\bar{x})$. \\
	The assumptions on $\kappa$ imply that 
	\begin{eqnarray}
		\forall u\in L^1(\Omega), \; u\geq 0  \; \exists 								\Omega'\subset\Omega \,\forall x\in\Omega': 								\quad 	\integral \kappa(x,y) u(y) \,\D y\leq 								u(x).\label{positmeas}
	\end{eqnarray}
	Assuming otherwise, 
	\begin{equation*}
	 \int_{\Omega}\kappa(x,y) u(y) \,\D y -u(x)
> 0,
\end{equation*}
for almost all $x$, and
integrating the left hand side with respect to $x$, using
$\int_{\Omega} \kappa(x,y)\mbox{d}x = 1$, we obtain 
$\int_{\Omega} \int_{\Omega} \kappa(x,y)u(y) \mbox{d}y
-u(x) \mbox{d} x = 0$, what implies a contradiction.

Taking $f_1(x) := \bigchi_{\Omega'}\varphi_{\varepsilon,\rho}(x)$ where $\varphi_{\varepsilon,\rho}$ is the positive eigenfunction corresponding to the eigenvalue $\lambda_{\varepsilon}(\rho)$ 	 and $\Omega'$ being a set of positive measure such that \eqref{positmeas} holds for $\varphi_{\varepsilon,\rho}$,
	\[\begin{array}{rcl}\lambda_{\varepsilon}(\rho)f_{1}(x)& =& \left(\frac{2a(x)}{1+k\rho}-1\right)pf_{1}(x)-\varepsilon pf_{1}(x)+\varepsilon p \int_{\Omega'}\kappa(x,y)f_{1}(y)\mbox{d}y\\ \\ & \leq & \left(\frac{2a(x)}{1+k\rho}-1\right)pf_{1}(x).
	\end{array}\]
	Hence, by \eqref{car} we conclude that $\lambda_{\varepsilon}(\rho)\leq \mu(x) \leq \mu(\bar{x})$.
	
	We show now that for all $\delta >0$, there exists $\varepsilon_{0}$ such that it holds
	$$
	\lambda_{\varepsilon}(\rho)\geq \mu(\bar{x})-\delta
	$$
	for any $\varepsilon < \varepsilon_{0}$.
	
	Let $\delta$ be such that $\mu_{\varepsilon}(x)\geq \mu_{\varepsilon}(\bar{x})-\delta$ for $x\in\Omega'$, where $\Omega'\subset \Omega$ is a suitably chosen 		set. Let us consider a smooth function $u$ with $\supp{u} \subset \Omega'$. Then,  	due to 			the positivity of $\eps{K}$, it holds
	\[C_{\varepsilon,\rho} u = B_{\varepsilon,\rho} 							u + K_{\varepsilon} u \geq \eps{\mu}(x)u 			\geq (\eps{\mu}(\bar{x}) - \delta)u.\]
	By inequality \eqref{car}, we obtain 				$\lambda_{\varepsilon}(\rho) \geq \eps{\mu}(\bar{x})-\delta$. Furthermore, we know
	\[\left| \mu(x) - \eps{\mu}(x) \right| \leq \varepsilon\leq \delta.			\]
	Thus, we conclude
	\[\forall \delta> 0\;\exists\; \varepsilon_0> 0 \;\forall 					\varepsilon<\varepsilon_0: \quad 											\lambda_{\varepsilon}(\rho) \in B_{2\delta}(\mu(\bar{x})).\]
\end{proof}

\begin{prop}\label{prop:mut:eifunconv}
	Let $\varphi_{\varepsilon,\rho}$ be the unique positive 				eigenfunction corresponding to the eigenvalue $\lambda_{\varepsilon}(\rho)$ of the operator $C_{\varepsilon,			\rho}$ defined in \eqref{operators}. Then 
	\[\varphi_{\varepsilon,\rho} \rightharpoonup^* 								\delta_{\bar{x}}\quad\mbox{in $\mathcal{M}^+(\Omega)$},\quad  \mbox{as}\; \varepsilon\rightarrow 0,\]
	where $\bar{x}$ is the unique value, where the maximum of the self-renewal function $a(x)$  is attained.
\end{prop}
\begin{proof}
	To prove convergence of the steady state, we make an ansatz that the eigenfunctions form a Dirac sequence.\\
	By Proposition \ref{prop:mut:existe2} and \cite[Theorem $2.2$]{burgerpert}, the eigenfunction 					$\varphi_{\varepsilon,\rho}$ of $C_{\varepsilon,\rho}$ is strictly 			positive, $\varphi_{\varepsilon,\rho} >0$ for all 					$\varepsilon>0$.  As an eigenfunction in 					$L^1(\Omega)$, it can be normalized to $\left\|\varphi_{\varepsilon,			\rho} \right\|_{L^1(\Omega)} = 1$. It remains to show that 
	\[\int\limits_{\Omega^c} \varphi_{\varepsilon,\rho}(x)\,\D x 				\stackrel{\varepsilon\rightarrow 0}{\longrightarrow} 0\]
	for 				$\Omega^c\subset \Omega$ with $\bar{x}\notin \Omega^c$ 			and 		$\mathrm{dist}(\bar{x},\Omega^c)>0$.
	
	According to Lemma \ref{lem:mut:eigenconv}, it holds							$\lambda_{\varepsilon}(\rho) \rightarrow \max_{x\in\Omega} \left( 			\frac{2a(x)}{1+k\rho}-1\right)p$ for $\varepsilon\rightarrow 0$. 			Hence, for any $\Omega^c$ as defined	above, it is possible to choose 			$\varepsilon<\varepsilon_0$ such that
	\begin{eqnarray}\label{ineq:mut:max}
		\forall x\in\Omega^c: \quad \lambda_{\varepsilon}(\rho) > \left( 		\frac{2a(x)}{1+k\rho}-1\right)p.
	\end{eqnarray}
	Integrating the eigenvalue problem for $C_{\varepsilon,\rho}$, 	we obtain
	$$
	\begin{array}{rcl}
		0 &=& \int\limits_{\Omega^c} \left( \frac{2a(x)}{1+k\rho} - 				(1+ \varepsilon)\right)p \varphi_{\varepsilon,\rho}				(x) - \lambda_{\varepsilon}(\rho) \varphi_{\varepsilon,\rho}				(x)\,\D x \\ \\ & &  + 																\varepsilon p\int\limits_{\Omega^c} \integral \kappa(x,y) 					\varphi_{\varepsilon,\rho}(y)\,\D y\,\D x \\
		&\leq& \left( \max\limits_{x\in\Omega^c}\left(\frac{2a(x)}					{1+k\rho}-1\right)p - \lambda_{\varepsilon}									(\rho)\right)\int\limits_{\Omega^c}\varphi_{\varepsilon,\rho}				(x)\,\D x +\\ \\ & &	\varepsilon p\int\limits_{\Omega^c}\integral 						\kappa(x,y)\varphi_{\varepsilon,\rho}(y)\,\D y\,\D x.
	\end{array}
	$$
	By inequality \eqref{ineq:mut:max}, the first term is negative and 			bounded, hence by rearranging the inequality we obtain, for a constant 			$C>0$,
	\[C\int\limits_{\Omega^c} \varphi_{\varepsilon,\rho}(x)\,\D x \leq 			\varepsilon\int\limits_{\Omega^c}\integral 									\kappa(x,y)\varphi_{\varepsilon,\rho}(y)\,\D y\,\D x \leq 					\varepsilon. \]
Consequently, $\varphi_{\varepsilon,\rho}$ is a Dirac sequence and 					converges subsequently to a Dirac measure concentrated in 				$\bar{x}$. 
\end{proof}

\subsection{Convergence of the steady states}
Now we prove weak$^*$ convergence of the steady states $(\rho_{\varepsilon}, u_{\varepsilon})$ for $\varepsilon\rightarrow 0$.  As a first step, we show 

\begin{prop}\label{prop:mut:rhoconv}
	Let $\eps{\rho}$ be the unique zero of $\eps{\lambda}(\rho)$, hence 		the first component of the steady state of system \eqref{eq:mut:sys2} and let $\bar{\rho_{0}} := \frac{2\bar{a}-1}{k}$. Then
	\[\eps{\rho} \stackrel{\varepsilon\rightarrow 0}{\longrightarrow} 			\bar{\rho}_0.\]
\end{prop}
\begin{proof}
	
Notice that $\bar{\rho_0}$ is the unique zero of the decreasing function $\lambda_{0}(\rho):=\max\limits_{x\in\overline{\Omega}} \left( \frac{2a(x)}{1+k\rho}-1\right)p$. Lemma \ref{lem:mut:eigenconv} and the fact that $\lambda_{0}(\rho)$ changes sign (recall that, by Assumption \ref{ass:mut:2} there exists $x_{*}\in \Omega$ such that 					  $a(x_{*})>\frac12$) prove the statement.

\end{proof}
The next result provides convergence of the family of steady states of system \eqref{eq:mut:sys2}.
\begin{theorem}\label{thm:mut:conv}
	Let $(\eps{\rho}, \eps{u})$ be the family of stationary solutions of System				\eqref{eq:mut:sys2}. Then 
	\begin{align*}
		\eps{u} \rightharpoonup^* 													\bar{\rho}_1\delta_{\bar{x}}\quad\mbox{in 									$\mathcal{M}^+(\Omega)$},\quad \eps{\rho} \rightarrow 						\bar{\rho}_0\quad\mbox{in $\R$},
	\end{align*}
	where $\bar{\rho}_0 = \frac{2\bar{a}-1}{k}$, 							  $\bar{\rho}_1 = \frac{d}{p}\frac{2\bar{a}-1}{k}$ and $\bar{x}$ is the isolated point where the maximum of the self-renewal function $a(x)$  is attained $\bar{x}=\arg  \max\limits_{x\in \Omega}a(x)$.
\end{theorem}
\begin{proof}
	Convergence of $\eps{\rho}$ has already been proven in Proposition 			\ref{prop:mut:rhoconv}. \\
	Convergence of $\eps{u}$ is done in two steps. By construction 			$\eps{u} = \eps{c}\varphi_{\varepsilon,\eps{\rho}}$,
	where $\varphi_{\varepsilon,\eps{\rho}}$ is the unique (normalized) eigenfunction corresponding to the zero eigenvalue of the operator $C_{\varepsilon,			\rho_{\varepsilon}}$ defined in \eqref{operators} and $c_{\varepsilon} = \frac{d\rho_{\varepsilon}}{\integral 2\left( 1 - \frac{a(x)}{1+k\rho_{\varepsilon}}\right)p \varphi_{\varepsilon,\rho_{\varepsilon}}(x)\,\D x}.$  
	\\Thus, we want to prove that both, the constants $\eps{c}$ and the eigenfunctions $\varphi_{\varepsilon,\eps{\rho}}$, converge.

	 The same 				argument as in Proposition 			\ref{prop:mut:eifunconv} yields convergence of 	$\varphi_{\varepsilon,\eps{\rho}}$ to the 				Dirac delta located at 	$\bar{x}$. Because of this property of 				$\varphi_{\varepsilon,\eps{\rho}}$ and convergence of 					$\eps{\rho}$, it follows
	\[\eps{c} = \frac{d\eps{\rho}}{\integral 2\left( 1 - \frac{a(x)}			{1+k\eps{\rho}}\right)p \varphi_{\varepsilon,\eps{\rho}}\,\D x} 			\stackrel{\varepsilon\rightarrow 0}{\longrightarrow} 						\frac{d\bar{\rho_0}}{2p\left( 1- \frac{\bar{a}}{1+k\bar{\rho_0}}\right)} 		= \bar{\rho}_1.\]
	 This 				concludes the proof. 
\end{proof}

We can finally formulate the result giving the behavior for small mutation rate of the equilibria of System \eqref{eq:mut:sys}.
\begin{theorem}
Let $(v_{\varepsilon}, u_{\varepsilon})$ be the family of stationary solutions of System \eqref{eq:mut:sys}. Then, for $\varepsilon\rightarrow 0$,
 \begin{align*}
	(\eps{v}, \eps{u})\rightharpoonup^* 													(\bar{\rho}_0\delta_{\bar{x}} , 													\bar{\rho}_1\delta_{\bar{x}})\quad\mbox{in 									$\mathcal{M}^+(\Omega)$},
	\end{align*}
	where $\bar{\rho}_0 = \frac{2\bar{a}-1}{k}$, 							  $\bar{\rho}_1 = \frac{d}{p}\frac{2\bar{a}-1}{k}$ and $\bar{x}$ is the unique value where the maximum of the self-renewal function $a(x)$  is attained.
\end{theorem}
\begin{proof}
Theorem \ref{thm:mut:conv} provides convergence of the second component of the steady state. The first component can be written as
$$
v_{\varepsilon}(x)=\frac{2}{d}\left(1-\frac{a(x)}{1+k \rho_{\varepsilon}}\right)pu_{\varepsilon}(x)=:g_{\varepsilon}(x)u_{\varepsilon}(x).
$$
Proposition \ref{prop:mut:rhoconv} implies that for any $f \in C_{c}$
\[g_{\varepsilon}(x)f(x)\stackrel{\varepsilon\rightarrow 0}{\longrightarrow} 
\frac{2}{d}\left(1-\frac{a(x)}{1+k \bar{\rho_{0}}}\right)pf(x)=:g_{0}(x)f(x) \quad \text{strongly},\] which, by Proposition $3.13$ in \cite{brezis}, implies that
$$
\langle g_{\varepsilon}(x)f(x),u_{\varepsilon}(x)\rangle \stackrel{\varepsilon\rightarrow 0}{\longrightarrow} \langle g_{0}(x)f(x), \bar{\rho}_1\delta_{\bar{x}}\rangle,
$$ 
that is,
\begin{align*}
	\eps{v} \rightharpoonup^* g_{0}(\bar{x}) \bar{\rho}_1\delta_{\bar{x}}=\frac{2\bar{a}-1}{k}\delta_{\bar{x}}
\end{align*}

which concludes the proof.
\end{proof}
\section{Stability of the steady states}
Selection-mutation equations can be written, in a general way, in the form
\begin{eqnarray}\label{eq:mut:shortsys}
	\frac{\partial}{\partial t}z(t,x) = A_{\varepsilon}(F(z))z,
	\end{eqnarray}
	 with $F$ being a linear function from the state space to an $m$-dimensional space and such that $A_{\varepsilon}(E)$ is a linear operator for a fixed $E=F(z)$.\\

Assuming that equation \eqref{eq:mut:shortsys} has a semilinear structure and the spectral mapping property holds (i.e., the growth bound of a semigroup is equal to the spectral bound of its generator, which is the case in $L^{1}$), the principle of linearised stability \cite{webb85,henry} yields local asymptotic stability of a steady state if the spectrum of the corresponding linearisation is located entirely in the open left half plane.
A stability result for equation \eqref{eq:mut:shortsys} is provided in Ref. \cite{calsina07}. It is shown using the principle of linearised stability and the fact
that, in case of  finite dimensional nonlinearity, the linearised operator at the steady state is a degenerated perturbation of a known operator with spectral bound equal to 0. This reduces
the computation of the spectrum of the linearisation to the computation of zeroes of the so-called Weinstein-Aronszajn determinant (\cite{kato}).\\

System \eqref{eq:mut:sys2} can be written in form \eqref{eq:mut:shortsys} with 
\begin{equation}
\label{F}
F:\R\times L^1(\Omega) \rightarrow \R, (\rho,u) \mapsto \rho,
\end{equation}
and 
\begin{eqnarray*}
		&A_{\varepsilon}\left(F\begin{pmatrix} \rho\\ u \end{pmatrix}\right) 		= A_{\varepsilon}(\rho)= \begin{pmatrix}
	-d & \integral 2\left(1 - \frac{a(x)}{1+k\rho}\right)p\cdot \,\D x 			\\
	0 & \left(\frac{2a(x)}{1+k\rho} -(1+ 										\varepsilon)\right)p  + \varepsilon p						\integral \kappa(x,y)\cdot\,\D y&
	\end{pmatrix}.
\end{eqnarray*}
Note that  operator $\eps{A}(\rho)$ generates a $C^0$ positive semigroup. Indeed,  $\eps{A}(\rho)$ can be written in the following way
\begin{align*}
	\eps{A}(\rho) &= \begin{pmatrix}
	-d & 0 \\
	0 & \left(\frac{2a(x)}{1+k\rho} -(1+\varepsilon)\right)p 	\cdot 
	\end{pmatrix} + \begin{pmatrix}
	0 & \integral 2\left(1 - \frac{a(x)}{1+k\rho}\right)p\cdot \,\D x \\
	0& \varepsilon p \integral \kappa(x,y)\cdot\,\D y
	\end{pmatrix},
	\end{align*}
where the first term is the generator of a $C^0$ positive semigroup and the second term is a linear positive operator, thus the sum generates a $C^0$ positive semigroup \cite{EN}.\\
Linearising system \eqref{eq:mut:sys2} at the steady state $z_{\varepsilon}:=(\rho_{\varepsilon},u_{\varepsilon})$, i.e., taking a perturbation $z=z_{\varepsilon}+\bar{z}$, applying the stationary equation $\frac{\partial}{\partial t}z_{\varepsilon}=0$ and Taylor's formula, we obtain
$$
\begin{array}{rcl}
\left(\begin{array}{c} \bar{\rho}'\\ \frac{\partial \bar{u}}{\partial t}\end{array} \right) & = & (\tilde{A}_{\varepsilon}+S_{\varepsilon})\left(\begin{array}{c} \bar{\rho}\\  \bar{u}\end{array} \right),
\end{array}
$$
where 
\begin{equation}
\label{eq:mut:aeps}\begin{array}{rcl}
		\tilde{A}_{\varepsilon} = A_{\varepsilon}(\rho_{\varepsilon}) &=&				\begin{pmatrix}
		-d & \integral 2\left( 1- \frac{a(x)}										{1+k\rho_{\varepsilon}}\right)p \cdot \,\D x 
	 \\
		0 & \left( \frac{2a(x)}{1+k\rho_{\varepsilon}}-(1+ 						\varepsilon)\right)p \cdot + 		\varepsilon p\integral 	\kappa(x,y)\cdot\,\D y
	\end{pmatrix},
	
	 \\
	 
	\eps{S} &= &\begin{pmatrix}
		\integral \frac{2a(x)kp\eps{u}}{(1+k\eps{\rho})^2} & 0 \\
		-\frac{2a(x)kp\eps{u}}{(1+k\eps{\rho})^2} & 0
	\end{pmatrix}.
	\end{array}
\end{equation}%
with $\tilde{A}_{\varepsilon}$, $S_{\varepsilon}$ defined in $\R \times L^{1}(\Omega)$.\\

We also define the following ``limit" operators in $\R \times \mathcal{M}^+(\Omega)$
\begin{equation}
\begin{array}{rcl}
\label{eq:mut:a0}
	\tilde{A}_0 =A_0(\bar{\rho_0}) &= \begin{pmatrix}
	-d & \integral 2\left( 1- \frac{a(x)}{1+k\bar{\rho_0}}\right)p \cdot 			\,\D x \\ 
	0 &  \left( \frac{2a(x)}{1+k\bar{\rho_0}} -1\right)p \cdot 
	\end{pmatrix}, \\
	S_0 &= \begin{pmatrix}
	 \frac{2\bar{a}kp\bar{\rho_1}}{(1+k\bar{\rho_0})^2} \, & 			0\\
	-\frac{2\bar{a}kp\bar{\rho_1}}{(1+k\bar{\rho_0})^2}\delta_{\bar{x}} & 0
	\end{pmatrix}.
\end{array}
\end{equation}
\\
where $\bar{\rho_0}$ and $\bar{\rho_1}$ are given by Theorem \ref{thm:mut:conv}.

 As mentioned before, our aim is to apply the stability result given in Ref. \cite{calsina07} to system \eqref{eq:mut:sys2}. For the sake of completeness, we summarize the two relevant theorems \cite[Theorem $1$ and $2$]{calsina07} into the following theorem (for $m=1$ which is the case for our model): 
\begin{theorem}\label{thm:mut:stability}
	Let $z_{\varepsilon}$ be a non-trivial positive steady state of equation\eqref{eq:mut:shortsys}, where $F$ is a linear function from the state space to an m-dimensional space and, for a fixed $E=F(z)$,  $A_{\varepsilon}(E)$ is a generator of a $C^0$ positive semigroup on the state space. Let 					$\tilde{A}_{\varepsilon} + S_{\varepsilon}$ be the linearisation of 		$A_{\varepsilon}$ at the equilibrium $z_{\varepsilon}$. Let 				$\omega_{\varepsilon}(\lambda)$, $\omega_0(\lambda)$ be the 					Weinstein-Aronszajn determinants for 										$\tilde{A}_{\varepsilon}+S_{\varepsilon}$ and $\tilde{A}_0 + S_0$, 			respectively and $D:=\left\{\lambda\in\C \left| \Re(\lambda) \geq 0, 			\lambda	\neq 0\right.\right\}$. Let $\omega_{\varepsilon}(\lambda),			\omega_0(\lambda)$ be holomorphic functions in $D$ such that 						$\omega_0(\lambda)$ does  not vanish in $D$ and 
	\begin{eqnarray}\label{ass:mut:det}
	\omega_{\varepsilon}(\lambda) \xrightarrow{\varepsilon\rightarrow 0}		 \omega_0(\lambda)
	\end{eqnarray}
	uniformly in $\lambda$ on compact sets in $D$. Additionally, assume 		that 
	\begin{eqnarray}\label{ass:mut:bound}
	\exists L> 0 \;\forall \left|\lambda\right| > L: \quad \left\| 				S_{\varepsilon} R(\lambda,\tilde{A}_{\varepsilon})\right\|_{\infty} < 	\frac12. 
	\end{eqnarray}
	If $0$ is a strictly dominant eigenvalue of 								$\tilde{A}_{\varepsilon}$ with algebraic multiplicity $1$, 					$P_{\varepsilon}$ is the projection onto the eigenspace of the 				eigenvalue $0$ and 
	\begin{eqnarray}\label{ass:mut:lim}
		F\left( P_{\varepsilon} S_{\varepsilon}z_{\varepsilon}\right) 				\neq 0\quad\mbox{and}\quad \liminf\limits_{(\varepsilon,					\lambda)\rightarrow (0^+,0)} \lambda F \left( 								(\tilde{A}_{\varepsilon} - \lambda)^{-1} 									S_{\varepsilon}z_{\varepsilon}\right) \neq 0,
	\end{eqnarray}
	then for $\varepsilon$ small enough the steady state 						$z_{\varepsilon}$ is locally asymptotically stable. 
\end{theorem}

\begin{theorem}\label{thm:mut:realstead}
	Let Assumption \ref{ass:mut:2} hold and additionally, let $\kappa$ 			be separable in its variables, i.e., 
	\[\exists \kappa_1,\kappa_2\in C(\overline{\Omega}):\; \kappa(x,y) = 	\kappa_1(x)\kappa_2(y).\]
 Then, for $\varepsilon$ small enough, the steady state $(\eps{\rho},\eps{u})$ 	of system \eqref{eq:mut:sys2} is locally 						asymptotically stable.
\end{theorem} 
The proof of this theorem is a direct application of Theorem \ref{thm:mut:stability}. Since it is technical, it is deferred to Appendix A.
\begin{theorem}\label{thm:mut:realstead2}
	Let Assumption \ref{ass:mut:2} hold and additionally, let $\kappa$ 			be separable in its variables, i.e. 
	\[\exists \kappa_1,\kappa_2\in C(\overline{\Omega}):\; \kappa(x,y) = 	\kappa_1(x)\kappa_2(y).\]
 Then, for $\varepsilon$ small enough, the steady state $(\eps{v},\eps{u})$ 	of system \eqref{eq:mut:sys} is locally asymptotically stable.
\end{theorem} 

\begin{proof}
By Theorem \ref{thm:mut:realstead}, it only remains to prove the result for the first component of the steady state $v_{\varepsilon}(x)$. It holds
$$
\frac{\partial}{\partial t} v_{\varepsilon}(t,x) = F(u_{\varepsilon}(t,x),\rho_{\varepsilon}(t)) 	- dv_{\varepsilon}(t,x)
$$
with $F(u_{\varepsilon}(t,x),\rho_{\varepsilon}(t)):=2\left( 				1-\frac{a(x)}{1+k\eps{\rho}(t)} \right)pu_{\varepsilon}(t,x)$ and 
$$
v_{\varepsilon}(t,x)=v_{0}(x)e^{-dt}+\int_{0}^{t}F(u_{\varepsilon}(\tau,x),\rho_{\varepsilon}(\tau))e^{-d(t-\tau)}\mbox{d}\tau.
$$
Then, since $(v_{\varepsilon}(x),u_{\varepsilon}(x))$ is an equilibrium of system \eqref{eq:mut:sys}, it follows
$$
\begin{array}{rcl}
\|v_{\varepsilon}(t,x)-v_{\varepsilon}(x)\|_{L^{1}(\Omega)}&=&\int_{\Omega}\big|v_{0}(x)e^{-dt}+\int_{0}^{t}F(u_{\varepsilon}(\tau,x),\rho_{\varepsilon}(\tau))e^{-d(t-\tau)}\mbox{d}\tau\\  \\ & & -\dfrac{F(u_{\varepsilon}(x),\rho_{\varepsilon})}{d}\big|\mbox{d}x \\ \\ & \leq & \int_{\Omega}|v_{0}(x)e^{-dt}|\mbox{d}x+ \int_{0}^{t}e^{-d(t-\tau)} \int_{\Omega}\big|F(u_{\varepsilon}(\tau,x),\rho_{\varepsilon}(\tau))\\ \\ & & -F(u_{\varepsilon}(x),\rho_{\varepsilon})\big|\mbox{d}x\mbox{d}\tau+ \int_{\Omega}\big|F(u_{\varepsilon}(x),\rho_{\varepsilon})\big|\mbox{d}x\\ \\& & \big(\int_{0}^{t}e^{-d(t-\tau)}\mbox{d}\tau-\frac{1}{d}\big)
\end{array}
$$
which tends to zero as $t \rightarrow \infty$ due to Theorem \ref{thm:mut:realstead} (for more details on the second term see the proof of Lemma 7 in \cite{busse2016}).
\end{proof}

\section{Appendix A}
Here we provide the proof of Theorem \ref{thm:mut:realstead}. The proof is divided into several parts, each dealing with a different assumption of the stability theorem \ref{thm:mut:stability}. 

\subsection{Convergence of the Weinstein-Aronszajn determinants}

The operators $S_{\varepsilon}$ and $S_{0}$ defined in \eqref{eq:mut:aeps} and \eqref{eq:mut:a0} are one-dimensional range operators with basis

\begin{align}\label{eq:mut:basis}
 \begin{pmatrix}
	\integral \frac{2a(x)kp\eps{u}}{(1+k\eps{\rho})^2} \,\D x \\
	-\frac{2a(x)kp\eps{u}}{(1+k\eps{\rho})^2}
\end{pmatrix} \quad \text{and} \quad \begin{pmatrix}
	\frac{2\bar{a}kp\bar{\rho_1} }{(1+k\bar{\rho_0})^2} \\
	-\frac{2\bar{a}kp\bar{\rho_1}}{(1+k\bar{\rho_0})^2}\delta_{\bar{x}},
\end{pmatrix}
\end{align}
respectively. 
Therefore the Weinstein-Aronszajn determinants  \begin{equation}\label{WA}\begin{array}{c}\eps{\omega}(\lambda):=\det\left( Id + \eps{S}R(\eps{\tilde{A}},\lambda)_{|						rg(\eps{S})}\right)\\ \omega_0(\lambda):= \det\left( Id + S_0R(\tilde{A}_0,\lambda)_{|			rg(S_0)}\right)\end{array}\end{equation}
are well defined. In the next lemma we prove convergence result  \eqref{ass:mut:det}. 
\begin{lemma}
	Let $\eps{\omega}(\lambda), \omega_0(\lambda)$ be the Weinstein-			Aronszajn 	determinants defined in \eqref{WA}. Then,
	\[\eps{\omega}(\lambda) \stackrel{\varepsilon\rightarrow 0}					{\longrightarrow} \omega_0(\lambda)\]
	uniformly in $\lambda \in D= \left\{\lambda\in\C\left| 						\Re(\lambda)\geq 0, \lambda\neq 0\right.\right\}$. Both 					$\eps{\omega}(\lambda)$ and $\omega_0(\lambda)$ are holomorphic in 			$D$. 
\end{lemma}
\begin{proof}
	In order to prove convergence, we estimate
	\begin{align*}
	\left| \eps{\omega}(\lambda) - \omega_0(\lambda)\right| &\leq \left| \det\left( Id + \eps{S}R(\eps{\tilde{A}},\lambda)_{|			rg(\eps{S})}\right) - \det\left( Id + \eps{S}R(\tilde{A}_0,					\lambda)_{|rg(\eps{S})}\right)\right| \\
	&  + \left| \det\left( Id + \eps{S}R(\tilde{A}_0,\lambda)_{|				rg(\eps{S})}\right) - \det\left( Id + S_0R(\tilde{A}_0,\lambda)_{|			rg(S_0)}\right)\right|.
	\end{align*}
	For the first term on the right-hand side it can be shown, in the same way as in \cite[proof of Proposition $1$]{silvia2} that
	\[\left\| \eps{S}(\eps{\tilde{A}} - \lambda)^{-1}_{|R(\eps{S})} - 			\eps{S}(\tilde{A}_0-\lambda)^{-1}_{|R(\eps{S})} \right\|_{\infty} 			\stackrel{\varepsilon\rightarrow 0}{\longrightarrow} 0.\]

	It remains to prove that 
	\[\left| \det\left( Id + \eps{S}R(\tilde{A}_0,\lambda)_{|					R(\eps{S})}\right) - \det\left( Id + S_0R(\tilde{A}_0,\lambda)_{|			R(S_0)}\right)\right| \stackrel{\varepsilon\rightarrow 0}					{\longrightarrow} 0.\]
	For this purpose we compute the determinant explicitly. The basis of 	$rg(\eps{S})$ is given by \eqref{eq:mut:basis}. Then, a direct computation yields
	\begin{eqnarray*}
		\eps{S}R(\tilde{A}_0,\lambda)_{|R(\eps{S})} &=& -\frac{1}					{d+\lambda} \integral \frac{2a(x)kp\eps{u}(x)}								{(1+k\eps{\rho})^2}\,\D x \\
		& &  + \frac{1}{d+\lambda} \integral 					2\left( 			1- \frac{a(x)}{1+k\bar{\rho_0}}\right)p \frac{1}								{\left(\frac{2a(x)}{1+k\bar{\rho_0}} - 1\right)p - \lambda} 					\frac{2a(x)kp\eps{u}(x)}{(1+k\eps{\rho})^2}\,\D x.
	\end{eqnarray*}
	According to Theorem \ref{thm:mut:conv}, we know that $\eps{\rho}$ 			converges strongly to $\bar{\rho_0}$ in $\R$ and $\eps{u}$ converges 			weakly$^*$ to $\bar{\rho_1}\delta_{\bar{x}}$ in $\mathcal{M}^+(\Omega)$. Hence, 
	\begin{eqnarray*}
		\eps{S}R(\tilde{A}_0,\lambda)_{|R(\eps{S})}								&\xrightarrow{\varepsilon\rightarrow 0}& -\frac{1}							{d+\lambda} \frac{2\bar{a}kp\bar{\rho_1}}{(1+k\bar{\rho_0})^2} - \frac{1}					{\lambda(d+\lambda)} 2\left( 1 - \frac{\bar{a}}								{1+k\bar{\rho_0}}\right)p\frac{2\bar{a}kp\bar{\rho_1}}										{(1+k\bar{\rho_0})^2} \\ 
		&=& S_{0}R(\tilde{A}_0,\lambda)_{|R(S_{0})}.
	\end{eqnarray*}
	By definition of the Weinstein-Aronszajn determinant, $\eps{\omega}			(\lambda)$ and $\omega_0(\lambda)$ are holomorphic in $D$, see 				\cite[p. 245]{kato}. 
\end{proof}
\subsection{Boundedness of $\eps{S}R(\lambda,\eps{\tilde{A}})$}
\begin{lemma}
	There exists a constant $L>0$ such that for all $\left| 					\lambda\right| > L$ 
	\[\left\| \eps{S} R(\lambda,\eps{\tilde{A}}) \right\|_{\infty} < 			\frac12.\]
\end{lemma}
\begin{proof}
	 Since $\sup_{\varepsilon<\varepsilon_0}\left\| 							\eps{\tilde{A}}\right\|_{\infty}$ and 										$\sup_{\varepsilon<\varepsilon_0}\left\|\eps{S}\right\|_{\infty}$ 			are bounded, we obtain for $\left|\lambda\right| > 2\left\| 					\eps{\tilde{A}} \right\|_{\infty}$
	\[\left\| \eps{S}R(\lambda, \eps{\tilde{A}}) \right\|_{\infty} = 			\left\| \eps{S} \lambda^{-1} \sum\limits_{n=0}^{\infty} 					\left(\lambda^{-1} 			\eps{\tilde{A}}\right)^n \right\|				_{\infty} \leq \frac{\left\| \eps{S} \right\|_{\infty}}{\lambda - 			\left\|\eps{\tilde{A}}\right\|_{\infty}} \leq \frac{2\left\| 				\eps{S}\right\|_{\infty}}{\left| \lambda\right|}.\]
	Choosing $L > \max\left\{ 2\left\| \eps{\tilde{A}}\right\|_{\infty}, 			4\left\| 			\eps{S} \right\|_{\infty}\right\}$ leads to the assertion.
\end{proof}
\subsection{Proof of hypotheses \eqref{ass:mut:lim} (excluding $0$ and values with small positive real part from the spectrum)}
\begin{lemma}\label{lem:mut:limits}
	For the steady state $z_{\varepsilon}:=(\rho_{\varepsilon},u_{\varepsilon}(x))$ of system \eqref{eq:mut:sys2}, it holds
	\begin{eqnarray}
		F\left( P_{\varepsilon} S_{\varepsilon}z_{\varepsilon}\right) 				\neq 0.
	\end{eqnarray}
\end{lemma}
\begin{proof}
From the definition of $F$ given in formula \eqref{F}, it is sufficient to show that $P_{\varepsilon} S_{\varepsilon}z_{\varepsilon} \neq 0$. Since $0$ is a simple strictly dominant eigenvalue of operator $\tilde{A}_{\varepsilon}$, we can decompose the space $L^{1}(\Omega)= \langle z_{\varepsilon} \rangle \bigoplus \text{Range}(\tilde{A}_{\varepsilon})$ (see Theorem A.3.1 in \cite{clement}). Hence, we have to prove that $S_{\varepsilon}z_{\varepsilon} \notin \text{Range}(\tilde{A}_{\varepsilon})$ what is equivalent to showing that
\begin{equation}
\label{adjoint}
\Big\langle \left(\begin{array}{c} \rho_{\varepsilon}^{*} \\
u_{\varepsilon}^{*}
\end{array}\right), S_{\varepsilon}\left(\begin{array}{c} \rho_{\varepsilon} \\
u_{\varepsilon}
\end{array}\right) \Big\rangle \neq 0,
\end{equation}
where $(\rho_{\varepsilon}^{*}, u_{\varepsilon}^{*})$ is the eigenfunction corresponding to the eigenvalue $0$ of the adjoint operator $\tilde{A}^{*}_{\varepsilon}$. The adjoint operator reads
$$
\begin{array}{rcl}
\tilde{A}_{\varepsilon}&=& \left(\begin{array}{cc}-d & 0  \\2\left( 1- \frac{a(x)}										{1+k\rho_{\varepsilon}}\right)p &  \left( \frac{2a(x)}{1+k\rho_{\varepsilon}}-(1+ 						\varepsilon)\right)p  + 		\varepsilon p\integral 	\kappa(y,x)\cdot\,\D y\end{array}\right)
\end{array}
$$
and we obtain that $\rho_{\varepsilon}^{*}=0$. This implies that $u_{\varepsilon}^{*}$ is an eigenfunction corresponding the the zero eigenvalue of the operator $\left( \frac{2a(x)}{1+k\rho_{\varepsilon}}-(1+ 						\varepsilon)\right)p \cdot + 		\varepsilon p\integral 	\kappa(y,x)\cdot$ which is the adjoint operator of $B_{\varepsilon,\rho}+K_{\varepsilon}$ defined by formulas \eqref{operators}. This operator is a generator of an irreducible positive semigroup in the Banach lattice $L^{1}(\Omega)$, as it is the perturbation by an irreducible operator of the generator of a positive semigroup. By Proposition 3.5 in \cite{nagel} we obtain that $u_{\varepsilon}^{*}$ is strictly positive, which, together with the fact that $\rho_{\varepsilon}^{*}=0$, implies that 
$$
\Big\langle \left(\begin{array}{c} \rho_{\varepsilon}^{*} \\
u_{\varepsilon}^{*}
\end{array}\right), S_{\varepsilon}\left(\begin{array}{c} \rho_{\varepsilon} \\
u_{\varepsilon}
\end{array}\right) \Big\rangle=-\int_{\Omega}\frac{2a(x)kpu_{\varepsilon}(x)\rho_{\varepsilon}u_{\varepsilon}^{*}(x)}{(1+k\rho_{\varepsilon})^{2}} \neq 0
$$

	\end{proof}
	The last step is to show
	\begin{lemma}\label{lem:mut:liminf}
	For the steady state $\eps{z}$, it holds
	\begin{align}\label{eq:mut:liminf}
	\liminf\limits_{(\varepsilon,\lambda)\rightarrow (0^+,0)} \lambda 			F \left((\tilde{A}				_{\varepsilon} - \lambda)^{-1} 				S_{\varepsilon}z_{\varepsilon}\right) \neq 0.
	\end{align}
	\end{lemma}
	\begin{proof}
	We start by computing the resolvent operator $R(						\eps{\lambda,\tilde{A}})$, 
	\[(\eps{\tilde{A}}-\lambda)^{-1} = \begin{pmatrix}
	\frac{-1}{d+\lambda} & \frac{1}{d+\lambda} \integral 2\left(1 - 			\frac{a(x)}{1+k\eps{\rho}}\right)p R(\lambda,C_{\varepsilon,				\eps{\rho}})\cdot \,\D x \\
	0 & R(\lambda,C_{\varepsilon,\eps{\rho}})
	\end{pmatrix},\]
	where recall that $C_{\varepsilon,\eps{\rho}}$ is defined in \eqref{operators}.
	Condition \eqref{eq:mut:liminf} reads
	$$
	\begin{array}{rcl}
	&\liminf\limits_{(\varepsilon,\lambda)\rightarrow (0^+,0)} \lambda 			F\left( (\eps{\tilde{A}}-\lambda)^{-1}\eps{S}\begin{pmatrix}
	\eps{\rho} \\ \eps{u}
	\end{pmatrix}\right)& \\
	= &\liminf\limits_{(\varepsilon,\lambda)\rightarrow 						(0^+,0)}\frac{-\lambda}{d+\lambda} \Big(\integral 						2\left(1 - \frac{a(x)}{1+k\eps{\rho}}\right)p 			R(\lambda,C_{\varepsilon,\eps{\rho}}) 										\frac{2a(x)kp\eps{u}\eps{\rho}}{(1+k\eps{\rho})^2} \,\D x & \\
& +  \integral 	\frac{2a(x)kp\eps{u}\eps{\rho}}{(1+k\eps{\rho})^2}\,			\D x\Big)  \neq 0.	 &
	\end{array}
	$$

Since the limit of the second term is zero, condition \eqref{eq:mut:liminf} becomes

	\[\liminf\limits_{(\varepsilon,\lambda) \rightarrow (0^+,0)} 				\frac{-1}{d+\lambda} \integral 						2\left(1 - 				\frac{a(x)}	{1+k\eps{\rho}}\right)p \lambda										R(\lambda,C_{\varepsilon,\eps{\rho}}) 										\frac{2a(x)kp\eps{u}\eps{\rho}}{(1+k\eps{\rho})^2} \,\D x \neq 0\]
 Determination of the limit 	is difficult. Since $0$ is an eigenvalue of $C_{\varepsilon,					\eps{\rho}}$, the limiting behaviour of $\lambda 							R(\lambda,C_{\varepsilon,\eps{\rho}})$ for $\lambda$ tending to zero 	is not obvious, because the resolvent tends to infinity ($C_{\varepsilon,\eps{\rho}}$ tends to a multiplication operator), while $\lambda$ tends to zero.  \\
However, separation of variables of the kernel $\kappa$ allows an explicit derivation of the resolvent $R(\lambda,C_{\varepsilon,\eps{\rho}})$ which facilitates the computation of the previous limit. Under this assumption, we write
\[C_{\varepsilon,\eps{\rho}}u= \left( \frac{2a(x)}{1+k\eps{\rho}} -(1+ \varepsilon)\right)p u + \varepsilon p \kappa_1(x)\integral \kappa_2(y)u(y)\,\D y = -\alpha_{\varepsilon}(x) u + \varepsilon\kappa_1(x)Lu,\]
where $\alpha_{\varepsilon}(x) := \left(1+ \varepsilon- \frac{2a(x)}{1+k\eps{\rho}}\right)p > 0$, since $\eps{\rho}$ is the first component of the steady state of system \eqref{eq:mut:sys2}, and $Lu = p\integral \kappa_2(y)u(y)\,\D y$. Following the scheme proposed in \cite[Section $4.2$]{calsina07} for the explicit computation of the resolvent operator $R(\lambda,C_{\varepsilon,\eps{\rho}})$, we obtain
\begin{align}\label{eq:resolvent}
R(\lambda,C_{\varepsilon,\eps{\rho}})g &= \frac{1}{\varepsilon\lambda} L\left(\frac{\kappa_1(x)}{\alpha_{\varepsilon}(x)		(\alpha_{\varepsilon}(x)+\lambda)}\right)^{-1}\Biggl[ -(\alpha_{\varepsilon}(x)+\lambda)^{-1} 			g\\
& + \varepsilon(\alpha_{\varepsilon}(x)+\lambda)^{-1} g 									L\left((\alpha_{\varepsilon}(x)+\lambda)^{-1}\kappa_1(x)\right) \\
	  &- \varepsilon(\alpha_{\varepsilon}(x)+\lambda)^{-1} \kappa_1(x) 						L\left((\alpha_{\varepsilon}(x)+ \lambda)^{-1} g\right)\Biggr].
\end{align}
	Let us define 
	\[\eps{\beta}(x) := \frac{2a(x)kp\eps{\rho}}									{(1+k\eps{\rho})^2},\; \eps{H}(x) := 2\left(1 - \frac{a(x)}					{1+k\eps{\rho}}\right)p,\]
	in order to shorten the notational effort. Note that
	\[\lim\limits_{\varepsilon\rightarrow 0} \eps{H}(x) = 2\left( 1 - 			\frac{a(x)}{1+k\bar{\rho_0}}\right)p =: H(x),\quad 							\lim\limits_{\varepsilon\rightarrow 0} \eps{\beta}(x) = 					\frac{2a(x)kp\bar{\rho_0}}{(1+k\bar{\rho_0})^2} =: \beta(x).\]
	We need to determine the following limiting process
	\begin{eqnarray*}
		\liminf\limits_{(\varepsilon,\lambda) \rightarrow (0^+,0)} 					\integral\eps{H}(x)\lambda R(\lambda,C_{\varepsilon,\eps{\rho}}) 		\eps{\beta}(x)\eps{u}(x)\,\D x =:\liminf\limits_{(\varepsilon,				\lambda)\rightarrow (0^+,0)} \Xi(\varepsilon,\lambda).
	\end{eqnarray*}
	Substituting the expression of the resolvent 					operator derived in \eqref{eq:resolvent}, we obtain
	\begin{eqnarray*}
	\Xi(\varepsilon,\lambda) &=& \integral \eps{H}(x)\frac{1}					{\varepsilon} L\left( \frac{\kappa_1(y)}{\alpha_{\varepsilon}(y)							(\alpha_{\varepsilon}(y)+\lambda)}\right)^{-1}\cdot \\
	& &  \Biggl[ -(\alpha_{\varepsilon}(x)+\lambda)^{-1} 			\eps{\beta}					(x)\eps{u}(x) +	\varepsilon (\alpha_{\varepsilon}(x)+\lambda)^{-1} 										\eps{\beta}(x)\eps{u}(x) L\left( \frac{\kappa_1(y)}								{\alpha_{\varepsilon}(y)+\lambda}\right) \\
	& & - \varepsilon (\alpha_{\varepsilon}(x)+\lambda)^{-1}\kappa_1(x) L\left( 				\frac{\eps{\beta}(y)\eps{u}(y)}{\alpha_{\varepsilon}(y)+\lambda}\right)\Biggr]\,\D 	x.
	\end{eqnarray*}
	For a better distinction between the terms, let us define
	\begin{eqnarray*}
		I &:=& -\integral \frac{1}{\varepsilon}\eps{H}(x) 							L\left(\frac{\kappa_1(y)}{\alpha_{\varepsilon}(y)											(\alpha_{\varepsilon}(y)+\lambda)}\right)^{-1} 											(\alpha_{\varepsilon}(x)+\lambda)^{-1} \eps{\beta}(x)\eps{u}(x)\,\D x, \\
		II &:=& \integral \eps{H}(x)(\alpha_{\varepsilon}(x)+\lambda)^{-1}\eps{\beta}				(x)\eps{u}(x)L\left(\frac{\kappa_1(y)}{\alpha_{\varepsilon}(y)							(\alpha_{\varepsilon}(y)+\lambda)}\right)^{-1} 											 L\left( \frac{\kappa_1(y)}													{\alpha_{\varepsilon}(y)+\lambda}\right) \,\D x, \\
		III &:=& -\integral \eps{H}(x)												(\alpha_{\varepsilon}(x)+\lambda)^{-1}\kappa_1(x) 										L\left(\frac{\kappa_1(y)}{\alpha_{\varepsilon}(y)											(\alpha_{\varepsilon}(y)+\lambda)}\right)^{-1} 											 L\left( \frac{\eps{\beta}(y)\eps{u}(y)}									{\alpha_{\varepsilon}(y)+\lambda}\right)\,\D x, \\
		\Xi(\varepsilon,\lambda) &=& I+II+III.
	\end{eqnarray*}

	Using the steady state equation 
	\begin{align}\label{eq:useful}
	\eps{u} = \varepsilon 						\frac{\kappa_1(x)}				{\alpha_{\varepsilon}(x)}L\eps{u},
	\end{align}
	we obtain
	\begin{eqnarray*}
		I &=& -\integral \eps{H}(x)\frac{\eps{\beta}								(x)\kappa_1(x)L\eps{u}}{\alpha_{\varepsilon}(x)(\alpha_{\varepsilon}(x)+\lambda)} L\left( 				\frac{\kappa_1(y)}{\alpha_{\varepsilon}(y)(\alpha_{\varepsilon}(y)+\lambda)}\right)^{-1}\,\D 		x \\
		&=& -L\eps{u} \integral \eps{H}(x) \frac{\eps{\beta}(x)}						{\kappa_2(x)} \frac{\kappa_1(x)\kappa_2(x)}{\alpha_{\varepsilon}(x)						(\alpha_{\varepsilon}(x)+\lambda)} L\left(\frac{\kappa_1(y)}{\alpha_{\varepsilon}(y)					(\alpha_{\varepsilon}(y)+\lambda)}\right)^{-1}\,\D x.
	\end{eqnarray*}
	The sequence $\eps{g}$ denoted by
	\[\eps{g}(x) := \frac{\kappa_1(x)\kappa_2(x)}{\alpha_{\varepsilon}(x)						(\alpha_{\varepsilon}(x)+\lambda)}L\left(\frac{\kappa_1(y)}{\alpha_{\varepsilon}(y)						(\alpha_{\varepsilon}(y)+\lambda)}\right)^{-1},\]
	defines a Dirac sequence. The definition also guarantees that				%
	\[\forall\,\varepsilon>0:\quad \eps{g}(x) > 0\;\mbox{and}\; 				\integral \eps{g}(x)\,\D x = 1.\]
	Let $\bar{x} = \mathrm{argmax}_{x\in\overline{\Omega}} a(x)$ and take $\Omega^c\subset\Omega$ such that $\bar{x}\notin 				\Omega^c$ and $\mathrm{dist}(\bar{x},\Omega^c)>0$.\\
	It follows from Theorem \ref{thm:mut:conv} that 
	\[\alpha_{\varepsilon}(x) = \left(1+ \varepsilon - \frac{2a(x)}				{1+k\eps{\rho}}\right)p \stackrel{\varepsilon\rightarrow 0}					{\longrightarrow} \left( 1 - \frac{2a(x)}{1+k\bar{\rho_0}}\right)p.\]
	We conclude that $\frac{\kappa_1(x)\kappa_2(x)}				{\alpha_{\varepsilon}(x)(\alpha_{\varepsilon}(x)+\lambda)}$ converges and is subsequently 				bounded on $\Omega^c$. \\
	Then, using $1 = \varepsilon 							L\left(\frac{\kappa_1(x)}{\alpha_{\varepsilon}(x)}\right)$, we estimate
	\begin{align*}
	\int\limits_{\Omega} \frac{\kappa_1(x)\kappa_2(x)}{\alpha_{\varepsilon}(x)				(\alpha_{\varepsilon}(x)+\lambda)}\,\D x &\geq \frac{1}									{\max\limits_{x\in\Omega}(\alpha_{\varepsilon}(x)+\lambda)} 								\int\limits_{\Omega} 														\frac{\kappa_1(x)\kappa_2(x)}{\alpha_{\varepsilon}(x)}\,\D x = \frac{1}					{\max\limits_{x\in\Omega} (\alpha_{\varepsilon}(x)+\lambda)} L\left( 						\frac{\kappa_1}{\alpha_{\varepsilon}}\right) \\
	&= \frac{1}{\varepsilon\max\limits_{x\in\Omega} 							(\alpha_{\varepsilon}(x)+\lambda)} \stackrel{\varepsilon\rightarrow 0}					{\longrightarrow} \infty.
	\end{align*}
	This implies that
	\[ \forall\, \Omega^c\subset \Omega, \bar{x}\notin \Omega^c, 				\mathrm{dist}(\bar{x},\Omega^c)>0:\quad 									\int\limits_{\Omega^c} \eps{g}(x)\,\D x \rightarrow 0\quad\mbox{for} 	\; \varepsilon\rightarrow 0.\]
	Since we additionally know by Theorem \ref{thm:mut:conv} that 				$\eps{u}$ 	converges weakly$^*$, we infer 
	\[L\eps{u} = \integral \kappa_2(y)\eps{u}(y)\,\D y \rightarrow 				\bar{\rho}_1\kappa_2(\bar{x})\quad\mbox{for}\; \varepsilon\rightarrow 0.			\]
	Thus, we obtain
	\[I \stackrel{\varepsilon\rightarrow 0}{\longrightarrow} 					-\bar{\rho}_1H(\bar{x})\beta(\bar{x}) < 0. \]

	Computing the limit for $II$ and $II$ and using equality \eqref{eq:useful}, we obtain

	$$
	\begin{array}{rcl}

		\liminf (II+III)&=& \liminf\limits_{(\varepsilon,\lambda)\rightarrow (0^+,0)} 				\varepsilon\Biggl[ \integral \eps{H}(x) 									\frac{L\eps{u}\eps{\beta}(x)}{\kappa_2(x)} 									\frac{\kappa_1(x)\kappa_2(x)}{\alpha_{\varepsilon}(x)(\alpha_{\varepsilon}(x)+\lambda)} 		\\	& &	L\left(\frac{\kappa_1(y)}{\alpha_{\varepsilon}(y)											(\alpha_{\varepsilon}(y)+\lambda)}\right)^{-1} \cdot 
		\left( L\left(\frac{\kappa_1(y)}{\alpha_{\varepsilon}(y)+\lambda}\right) - 		\frac{\alpha_{\varepsilon}(x)}{\eps{\beta}(x)} 									L\left(\frac{\eps{\beta}(y)\kappa_1(y)}{\alpha_{\varepsilon}(y)							(\alpha_{\varepsilon}(y)+\lambda)}\right)\right)\,\D x 		\Biggr] \\ &=&0

	\end{array}
	$$
	which concludes the proof. 
\end{proof}
\begin{remark}
Note that the separation of variables of $\kappa$ is needed only, because of the explicit computation of the resolvent $R(\lambda,C_{\varepsilon,\eps{\rho}})$. All results up to this point do not need this assumption and work for Assumption \ref{ass:mut:2} alone. 
\end{remark}

\section{Appendix B}
In this appendix we prove that the operators 
\[ \begin{array}{rcl}T^n_{\varepsilon,\rho} u(x)& =& \int\limits_{\Omega^n} 						\kappa(x,y_n)\prod\limits_{i=1}^{n-1} 										\kappa(y_i,y_{i+1})\\ & & \frac{(1+k\rho)^n}{\prod\limits_{i=1}^n 					((1+k\rho)(1+ \varepsilon)-2a(y_i))}u(y_n)\,\D y_1 			\cdots \D y_n.\end{array}\]
defined in the proof of Lemma \ref{lem:mut:specmon} satisfy $\tfrac{\D}					{\D\rho}T^n_{\varepsilon,\rho}<0$ for $\varepsilon$ small enough. The 					differential operator and the integral can be interchanged, because 		of Leibniz' integral rule. This implies denoting by $\D\vec{y}= \D 	y_1 \cdots \D y_n$
	\begin{align*}
		\frac{\D }{\D \rho} (T_{\varepsilon,\rho}^nu)(x) &= \frac{\D}				{\D\rho}\int\limits_{\Omega^n} 												\kappa(x,y_n)\prod\limits_{i=1}^{n-1} 										\kappa(y_i,y_{i+1})\frac{(1+k\rho)^n}{\prod\limits_{i=1}^n 					((1+k\rho)(1+ \varepsilon\hat{\kappa})-2a(y_i))p}u(y_n)\,\D 				\vec{y}\\
		&=\int\limits_{\Omega^n} \kappa(x,y_n)\prod\limits_{i=1}^{n-1} 				\kappa(y_i,y_{i+1})	\frac{\D}{\D\rho} \frac{(1+k\rho)^n}					{\prod\limits_{i=1}^n ((1+k\rho)											(1+\varepsilon\hat{\kappa})-2a(y_i))p} u(y_n)\,\D \vec{y}.
	\end{align*}
	Both $\kappa$ and $u$ are positive functions, so the sign of the 			derivative is solely determined by the derivative of the fraction. 			Performing the derivative yields
	\begin{align*}
		\frac{\D}{\D\rho} \frac{(1+k\rho)^n}{\prod\limits_{i=1}^n 					((1+k\rho)(1+ \varepsilon\hat{\kappa})-2a(y_i))p} &= 						\frac{kn(1+k\rho)^{n-1} 													\prod\limits_{i=1}^n((1+k\rho)(1+ \varepsilon\hat{\kappa}) - 				2a(y_i))p}{\left(\prod\limits_{i=1}^n 										((1+k\rho)(1+ \varepsilon\hat{\kappa})-2a(y_i))p\right)^2}\\
		&   - \frac{(1+k\rho)^n \frac{\D}			{\D\rho} 						\prod\limits_{i=1}^n((1+k\rho)(1+ 											\varepsilon\hat{\kappa})-2a(y_i))p}											{\left(\prod\limits_{i=1}^n ((1+k\rho)(1+ 									\varepsilon\hat{\kappa})-2a(y_i))p\right)^2}.
	\end{align*}
	Again we see that it is sufficient to look only at a small part of 			this derivative to determine the sign, namely the numerator. The 			claim is
	\[kn\prod\limits_{i=1}^n((1+k\rho)(1+ \varepsilon\hat{\kappa}) - 			2a(y_i))p - (1+k\rho) \frac{\D}												{\D\rho}\prod\limits_{i=1}^n((1+k\rho)										(1+\varepsilon\hat{\kappa})-2a(y_i))p 			< 0 \]
	and can be shown by induction over $n\in \N$. \\
	Let $n=1$, then 
	\[k((1+k\rho)(1+ \varepsilon\hat{\kappa}) -2a(y_1))p - (1+k\rho)(1+ 		\varepsilon\hat{\kappa})kp = -2kpa(y_1) < 									0,\]
	by Assumption \ref{ass:mut:2}.\\
	Let the statement be true for $n\in\N$. Then have a look at the 			derivative for $n+1$
	\begin{eqnarray*}
		& & k(n+1) \prod\limits_{i=1}^{n+1}((1+k\rho)(1+ 							\varepsilon\hat{\kappa}) - 2a(y_i))p\\
		& &  - 										(1+k\rho)\frac{\D}{\D \rho} \prod											\limits_{i=1}^{n+1} ((1+k\rho)(1+ \varepsilon\hat{\kappa}) - 				2a(y_i))p \\
		&=& kn \prod\limits_{i=1}^n((1+k\rho)(1+ \varepsilon\hat{\kappa}) - 2a(y_i))p \cdot ((1+k\rho)(1+ \varepsilon\hat{\kappa}) - 		2a(y_{n	+1}))p\\
		& &  + k\prod\limits_{i=1}^{n+1} ((1+k\rho)(1+ \varepsilon\hat{\kappa}) - 2a(y_i))p  		\\
		& & -(1+k\rho) \Biggl[ \frac{\D}{\D\rho} \prod\limits_{i=1}^n 				((1+k\rho)(1+ \varepsilon\hat{\kappa})- 2a(y_i))p \cdot 					((1+k\rho)(1+ \varepsilon\hat{\kappa}) - 2a(y_{n+1}))p \\
		& & + kp\prod\limits_{i=1}^n ((1+k\rho)(1+ 									\varepsilon\hat{\kappa}) - 2a(y_i))p\Biggr] \\
		&=& ((1+k\rho)(1+ \varepsilon\hat{\kappa}) - 2a(y_{n+1}))p \Biggl[ kn\prod\limits_{i=1}^n 				((1+k\rho)(1+ \varepsilon\hat{\kappa}) - 2a(y_i))p \\
		& & - (1+k\rho)\frac{\D}{\D\rho} \prod\limits_{i=1}^n ((1+k\rho)			(1+ \varepsilon\hat{\kappa}) - 		2a(y_i))p\Biggr]\\ 
	 	& & + k\Biggl[ \prod\limits_{i=1}^{n+1} ((1+k\rho)(1+ 						\varepsilon\hat{\kappa}) -2a(y_i))p \\
	 	& & - (1+k\rho)p \prod\limits_{i=1}^n ((1+k\rho)(1+ \varepsilon\hat{\kappa}) - 							2a(y_i))p\Biggr] < 0,
	\end{eqnarray*}
	because the first term is negative due to the induction assumption 			and the second term is negative because $1+k\rho > (1+k\rho)(1+ 			\varepsilon\hat{\kappa}) - 2a(y_{n+1})$ for $\varepsilon$ small 			enough.
\section*{Acknowledgements} S.C. has been partially supported by the grant
MTM2017-84214-C2-2-P from MICINN. Research of A.M.-C. and J.-E.B. has been part of SFB 873 supported by German Research Foundation (DFG).
\bibliography{Bibliography}

\begin{thebibliography}{10}

\bibitem{alberts2013}
B.~Alberts, D.~Bray, K.~Hopkin, A.~Johnson, J.~Lewis, M.~Raff, K.~Roberts, and
  P.~Walter.
\newblock {\em Essential cell biology}.
\newblock Garland Science, 4 edition, 2013.

\bibitem{Almeida}
L.~Almeida, P.~Bagnerini, G~Fabrini, B.D. Hughes, and T.~Lorenzi.
\newblock Evolution of cancer cell populations under cytotoxic therapy and
  treatment optimisation: insight from a phenotype-structured model.
\newblock {\em ESAIM Math. Model. Numer. Anal.}, 53:1157--1190, 2019.

\bibitem{nagel}
W.~Arendt, A.~Grabosch, G.~Greiner, U.~Groh, H.~P. Lotz, U.~Moustakas,
  R.~Nagel, F.~Neubrander, and U.~Schlotterbeck.
\newblock {\em One-parameter semigroups of positive operators}, volume 1184 of
  {\em Lecture Notes in Mathematics}.
\newblock Springer-Verlag, Berlin, 1986.

\bibitem{Bonnet}
D.~Bonnet and J.~E. Dick.
\newblock Human acute myeloid leukemia is organised as a hierarchy that
  originates from a primitive hematopoietic cell.
\newblock {\em Nat Med}, 3:730--7, 1997.

\bibitem{brezis}
Haim Brezis.
\newblock {\em Functional analysis, {S}obolev spaces and partial differential
  equations}.
\newblock Universitext. Springer, New York, 2011.

\bibitem{burgerpert}
R.~B{\"u}rger.
\newblock Perturbations of positive semigroups and applications to population
  genetics.
\newblock {\em Mathematische Zeitschrift}, 197(2):259--272, 1988.

\bibitem{burger}
R.~B\"{u}rger.
\newblock {\em The mathematical theory of selection, recombination, and
  mutation}.
\newblock Wiley Series in Mathematical and Computational Biology. John Wiley \&
  Sons, Ltd., Chichester, 2000.

\bibitem{BB}
R~B\"{u}rger and I.~M. Bomze.
\newblock Stationary distributions under mutation-selection balance: structure
  and properties.
\newblock {\em Adv. in Appl. Probab.}, 28(1):227--251, 1996.

\bibitem{busse2016}
J-E Busse, Piotr Gwiazda, and Anna Marciniak-Czochra.
\newblock Mass concentration in a nonlocal model of clonal selection.
\newblock {\em Journal of mathematical biology}, 73(4):1001--1033, 2016.

\bibitem{calsina2004}
{\`A}.~Calsina and S.~Cuadrado.
\newblock Small mutation rate and evolutionarily stable strategies in infinite
  dimensional adaptive dynamics.
\newblock {\em Journal of mathematical biology}, 48(2):135--159, 2004.

\bibitem{calsina05}
{\`A}.~Calsina and S.~Cuadrado.
\newblock Stationary solutions of a selection mutation model: The pure mutation
  case.
\newblock {\em Mathematical Models and Methods in Applied Sciences},
  15(07):1091--1117, 2005.

\bibitem{calsina07}
{\`A}.~Calsina and S.~Cuadrado.
\newblock Asymptotic stability of equilibria of selection-mutation equations.
\newblock {\em Journal of mathematical biology}, 54(4):489--511, 2007.

\bibitem{CCDR}
\`A Calsina, S.~Cuadrado, L.~Desvillettes, and G.~Raoul.
\newblock Asymptotics of steady states of a selection-mutation equation for
  small mutation rate.
\newblock {\em Proc. Roy. Soc. Edinburgh Sect. A}, 143(6):1123--1146, 2013.

\bibitem{ChisholmLorenzietal}
R.H. Chisholm, T.~Lorenzi, L.~Desvillettes, and B.D. Hughes.
\newblock Evolutionary dynamics of phenotype-structured populations: from
  individual-level mechanisms to population-level consequences.
\newblock {\em Z. angew. Math. Phys}, 67:1--34, 2016.

\bibitem{ChisholmLorenziClairambault}
R.H. Chisholm, T.~Lorenzi, L.~Desvillettes, and B.D. Hughes.
\newblock Tracking the evolution of cancer cell populations through the
  mathematical lens of phenotype-structured equations.
\newblock {\em Biology Direct}, 11:1--17, 2016.

\bibitem{clement}
Ph. Cl\'{e}ment, H.~J. A.~M. Heijmans, S.~Angenent, C.~J. van Duijn, and
  B.~de~Pagter.
\newblock {\em One-parameter semigroups}, volume~5 of {\em CWI Monographs}.
\newblock North-Holland Publishing Co., Amsterdam, 1987.

\bibitem{silvia1}
S.~Cuadrado.
\newblock Equilibria of a predator prey model of phenotype evolution.
\newblock {\em Journal of Mathematical Analysis and Applications},
  354(1):286--294, 2009.

\bibitem{silvia2}
S.~Cuadrado.
\newblock Stability of equilibria of a predator prey model of phenotype
  evolution.
\newblock {\em Math. Biosci. Eng}, 6:701--718, 2009.

\bibitem{daners92}
D.~Daners and P.~K. Medina.
\newblock {\em Abstract evolution equations, periodic problems and
  applications}, volume 279.
\newblock Chapman \& Hall/CRC, 1992.

\bibitem{diekmann2005}
Odo Diekmann, Pierre-Emanuel Jabin, St{\'e}phane Mischler, and Beno{\i}t
  Perthame.
\newblock The dynamics of adaptation: an illuminating example and a
  hamilton--jacobi approach.
\newblock {\em Theoretical population biology}, 67(4):257--271, 2005.

\bibitem{Ding}
L.~Ding, T.~J. Ley, D.~E. Larson, C.~A. Miller, D.~C. Koboldt, J.~S. Welch, and
  J.~F. DiPersio.
\newblock Clonal evolution in relapsed acute myeloid leukaemia revealed by
  whole-genome sequencing.
\newblock {\em Nature}, 481:506--510, 2012.

\bibitem{EN}
Klaus-Jochen Engel and Rainer Nagel.
\newblock {\em A short course on operator semigroups}.
\newblock Universitext. Springer, New York, 2006.

\bibitem{eveson}
S.~P. Eveson.
\newblock Compactness criteria for integral operators in {$L^\infty$} and
  {$L^1$} spaces.
\newblock {\em Proc. Amer. Math. Soc.}, 123(12):3709--3716, 1995.

\bibitem{Getto}
P~Getto, A~Marciniak-Czochra, Y~Nakata, and M~Vivanco.
\newblock Global dynamics of two-compartment models for cell production systems
  with regulatory mechanisms.
\newblock {\em Mathematical Biosciences}, 245:258--268, 2013.

\bibitem{graw2015}
Jochen Graw.
\newblock {\em Genetik}.
\newblock Springer Spektrum, 6 edition, 2015.

\bibitem{GLGL}
J.~Greene, O.~Lavi, M.~M. Gottesman, and D.~Levy.
\newblock The impact of cell density and mutations in a model of multidrug
  resistance in solid tumors.
\newblock {\em Bull. Math. Biol.}, 76(3):627--653, 2014.

\bibitem{greiner}
G.~Greiner.
\newblock A typical {P}erron-{F}robenius theorem with applications to an
  age-dependent population equation.
\newblock In {\em Infinite-dimensional systems ({R}etzhof, 1983)}, volume 1076
  of {\em Lecture Notes in Math.}, pages 86--100. Springer, Berlin, 1984.

\bibitem{henry}
D.~Henry.
\newblock {\em Geometric theory of semilinear parabolic equations}, volume 840
  of {\em Lecture Notes in Mathematics}.
\newblock Springer-Verlag, Berlin-New York, 1981.

\bibitem{Hope}
K.~J. Hope, L.~Jin, and J.~E. Dick.
\newblock Acute myeloid leukemia originates from a hierarchy of leukemic stem
  cell classes that differ in self-renewal capacity.
\newblock {\em Nat Immunology}, 5:738--43, 2004.

\bibitem{Jung}
N~Jung, B~Dai, AJ~Gentles, R~Majeti, and AP. Feinberg.
\newblock An {LSC} epigenetic signature is largely mutation independent and
  implicates the {HOXA} cluster in {AML} pathogenesis.
\newblock {\em Nature Communications}, 6:8489, 2015.

\bibitem{kato}
T.~Kato.
\newblock {\em Perturbation theory for linear operators}.
\newblock Springer Science \& Business Media, 1984.

\bibitem{kimura}
M.~Kimura.
\newblock A stochastic model concerning the maintenance of genetic variability
  in quantitative characters.
\newblock {\em Proceedings of the National Academy of Sciences of the United
  States of America}, 54:731--736, 1965.

\bibitem{Knauer}
F.~Knauer, T.~Stiehl, and A.~Marciniak-Czochra.
\newblock Oscillations in a white blood cell production model with multiple
  differentiation stages.
\newblock {\em J. Math. Biol.}, 80:576--600, 2020.

\bibitem{Kondo}
S~Kondo, S~Okamura, Y~Asano, M~Harada, and Y~Niho.
\newblock Human granulocyte colony-stimulating factor receptors in acute
  myelogenous leukemia.
\newblock {\em European Journal of Haematology}, 46:223--230, 1991.

\bibitem{Layton}
JE~Layton, H~Hockman, WP~Sheridan, and G.~Morstyn.
\newblock Evidence for a novel in vivo control mechanism of granulopoiesis:
  mature cell-related control of a regulatory growth factor.
\newblock {\em Blood}, 74:1303--1307, 1989.

\bibitem{LLR}
T.~Lorenzi, A.~Lorz, and G.~Restori.
\newblock Asymptotic dynamics in populations structured by sensitivity to
  global warming and habitat shrinking.
\newblock {\em Acta Appl. Math.}, 131:49--67, 2014.

\bibitem{LMCS}
T.~Lorenzi, A.~Marciniak-Czochra, and T.~Stiehl.
\newblock Mathematical modeling of leukemogenesis and cancer stem cell
  dynamics.
\newblock {\em J. Math. Biol.}, 79:1587--1621, 2019.

\bibitem{lorz2013}
A.~Lorz, T.~Lorenzi, M.~E. Hochberg, J.~Clairambault, and B.~Perthame.
\newblock Populational adaptive evolution, chemotherapeutic resistance and
  multiple anti-cancer therapies.
\newblock {\em ESAIM: Mathematical Modelling and Numerical Analysis},
  47(2):377--399, 2013.

\bibitem{lorz}
A.~Lorz, S.~Mirrahimi, and B.~Perthame.
\newblock Dirac mass dynamics in multidimensional nonlocal parabolic equations.
\newblock {\em Communications in Partial Differential Equations},
  36(6):1071--1098, 2011.

\bibitem{Lutz}
C.~Lutz, V.~T. Hoang, E.~Buss, and A.~D. Ho.
\newblock Identifying leukemia stem cells - is it feasible and does it matter?
\newblock {\em Cancer Lett}, 338:10--14, 2012.

\bibitem{Mikelic}
A.~Marciniak-Czochra, A.~Mikelic, and T.~Stiehl.
\newblock Renormalization group second order approximation for singularly
  perturbed nonlinear ordinary differential equations.
\newblock {\em Mathematical Methods in the Applied Sciences}, 41:5691--5710,
  2018.

\bibitem{Marciniak2009}
A.~Marciniak-Czochra, T.~Stiehl, A.~D. Ho, W.~J{\"a}ger, and W.~Wagner.
\newblock Modeling of asymmetric cell division in hematopoietic stem
  cells-regulation of self-renewal is essential for efficient repopulation.
\newblock {\em Stem cells and development}, 18(3):377--386, 2009.

\bibitem{Metzeler}
KH~Metzeler, K~Maharry, J~Kohlschmidt, S~Volinia, K~Mrozek, H~Becker,
  D~Nicolet, SP~Whitman, JH~Mendler, S~Schwind, AK~Eisfeld, YZ~Wu, BL~Powell,
  TH~Carter, M~Wetzler, JE~Kolitz, MR~Baer, AJ~Carroll, RM~Stone, MA~Caligiuri,
  G~Marcucci, and CD. Bloomfield.
\newblock A stem cell-like gene expression signature associates with inferior
  outcomes and a distinct micro{RNA} expression profile in adults with primary
  cytogenetically normal acute myeloid leukemia.
\newblock {\em Leukemia}, 27(10):2023--2031, 2013.

\bibitem{mirrahimi}
S.~Mirrahimi.
\newblock Adaptation and migration of a population between patches.
\newblock {\em Discrete \& Continuous Dynamical Systems-Series B}, 18(3), 2013.

\bibitem{Nakata}
Y~Nakata, P~Getto, A~Marciniak-Czochra, and T~Alarcon.
\newblock Stability analysis of multi-compartment models for cell production
  systems.
\newblock {\em Journal of Biological Dynamics}, 6 Suppl 1:2--18, 2012.

\bibitem{pazy}
A.~Pazy.
\newblock {\em Semigroups of linear operators and applications to partial
  differential equations}, volume~44 of {\em Applied Mathematical Sciences}.
\newblock Springer-Verlag, New York, 1983.

\bibitem{diplom}
B.~Perthame and G.~Barles.
\newblock Dirac concentrations in lotka-volterra parabolic pdes.
\newblock {\em Indiana University Mathematics Journal}, 57(7):3275--3301, 2008.

\bibitem{schaefer}
Helmut~H. Schaefer.
\newblock {\em Banach lattices and positive operators}.
\newblock Springer-Verlag, New York-Heidelberg, 1974.
\newblock Die Grundlehren der mathematischen Wissenschaften, Band 215.

\bibitem{Shinjo}
K~Shinjo, A~Takeshita, K~Ohnishi, and R~Ohno.
\newblock Granulocyte colony-stimulating factor receptor at various stages of
  normal and leukemic hematopoietic cells.
\newblock {\em Leukemia \& Lymphoma}, 25:37--46, 1997.

\bibitem{StiehlBaranHoMarciniak}
T.~Stiehl, N.~Baran, A.~D. Ho, and A.~Marciniak-Czochra.
\newblock Clonal selection and therapy resistance in acute leukemias:
  Mathematical modelling explains different proliferation patterns at diagnosis
  and relapse.
\newblock {\em J. Royal Society Interface}, 11, 2014.

\bibitem{StiehlBaranHoMarciniakCR}
T~Stiehl, N~Baran, AD~Ho, and A~Marciniak-Czochra.
\newblock Cell division patterns in acute myeloid leukemia stem-like cells
  determine clinical course: a model to predict patient survival.
\newblock {\em Cancer Research}, 75:940--949, 2015.

\bibitem{StiehlBMT}
T~Stiehl, AD~Ho, and A~Marciniak-Czochra.
\newblock The impact of {CD34+} cell dose on engraftment after {SCTs}:
  personalized estimates based on mathematical modeling.
\newblock {\em Bone Marrow Transplant}, 49:30--37, 2014.

\bibitem{StiehlSciRep}
T~Stiehl, AD~Ho, and A~Marciniak-Czochra.
\newblock Cytokine response of leukemic cells has impact on patient prognosis:
  Insights from mathematical modeling.
\newblock {\em Scientific Reports}, 8:2809, 2018.

\bibitem{StiehlBD}
T~Stiehl, C~Lutz, and A.~Marciniak-Czochra.
\newblock Emergence of heterogeneity in acute leukemias.
\newblock {\em Biology Direct}, 11(1):51, 2016.

\bibitem{StiehlMCM}
T~Stiehl and A~Marciniak-Czochra.
\newblock Characterization of stem cells using mathematical models of
  multistage cell lineages.
\newblock {\em Mathematical and Computer Modelling}, 53:1505--1517, 2011.

\bibitem{StiehlMMNP}
T.~Stiehl and A.~Marciniak-Czochra.
\newblock Mathematical modeling of leukemogenesis and cancer stem cell
  dynamics.
\newblock {\em Mathematical Modelling of Natural Phenomena}, 7(1):166--202,
  2012.

\bibitem{SMC-Review2017}
T~Stiehl and A.~Marciniak-Czochra.
\newblock Stem cell self-renewal in regeneration and cancer: Insights from
  mathematical modeling.
\newblock {\em Current Opinion in Systems Biology}, 5:112--120, 2017.

\bibitem{SMC-Review2019}
T~Stiehl and A.~Marciniak-Czochra.
\newblock How to characterize stem cells? contributions from mathematical
  modeling. doi: 10.1007/s40778-019-00155-0.
\newblock {\em Current Stem Cell Reports}, 2019.

\bibitem{VanDelft}
F.~W. Van~Delft, S.~Horsley, S.~Colman, K.~Anderson, C.~Bateman, H.~Kempski,
  J.~Zuna, C.~Eckert, V.~Saha, L.~Kearney, et~al.
\newblock Clonal origins of relapse in etv6-runx1 acute lymphoblastic leukemia.
\newblock {\em Blood}, 117:6247--54, 2011.

\bibitem{Wang}
W~Wang, T~Stiehl, S~Raffel, VT~Hoang, I~Hoffmann, L~Poisa-Beiro, BR~Saeed,
  R~Blume, L~Manta, V~Eckstein, T~Bochtler, P~Wuchter, M~Essers, A~Jauch,
  A~Trumpp, A~Marciniak-Czochra, AD~Ho, and C.~Lutz.
\newblock Reduced hematopoietic stem cell frequency predicts outcome in acute
  myeloid leukemia.
\newblock {\em Haematologica}, 102(9):1567--1577, 2017.

\bibitem{webb85}
G.~F. Webb.
\newblock {\em Theory of nonlinear age-dependent population dynamics}.
\newblock CRC Press, 1985.

\bibitem{weis}
L.~Weis.
\newblock The stability of positive semigroups on {$L_p$} spaces.
\newblock {\em Proc. Amer. Math. Soc.}, 123(10):3089--3094, 1995.

\bibitem{zeidler}
E~Zeidler.
\newblock {\em Nonlinear functional analysis and its applications. {I}}.
\newblock Springer-Verlag, New York, 1986.
\newblock Fixed-point theorems, Translated from the German by Peter R. Wadsack.

\end{thebibliography}

\end{document}